\documentclass{amsart}
\usepackage{amssymb,amstext,amsmath,amscd,amsthm,amsfonts,enumerate,graphicx,latexsym,stmaryrd,multicol}
\usepackage[usenames]{color}
\usepackage{setspace}
\usepackage{bm}
\usepackage[all]{xy}
\usepackage{soul}
\newtheorem{thm}{Theorem}[section]
\newtheorem{lem}[thm]{Lemma}
\newtheorem{prop}[thm]{Proposition}
\newtheorem{cor}[thm]{Corollary}
\theoremstyle{definition}
\newtheorem{dfn}[thm]{Definition}

\newtheorem{rem}[thm]{Remark}

\newtheorem{ex}[thm]{Example}

\theoremstyle{remark}

\newtheorem*{ac}{Acknowledgments}
\newtheorem*{proof of claim}{Proof of Claim}
\numberwithin{equation}{thm}
\def\Ann{\mathsf{Ann}}

\def\ca{\mathsf{ca}}

\def\cm{\mathsf{CM}}

\def\cok{\operatorname{Coker}}
\def\codim{\operatorname{codim}}

\def\dim{\operatorname{dim}}
\def\Ext{\operatorname{Ext}}

\def\height{\operatorname{ht}}
\def\Hom{\operatorname{Hom}}

\def\image{\operatorname{Im}}
\def\jac{\operatorname{jac}}
\def\ker{\operatorname{Ker}}
\def\m{\mathfrak{m}}
\def\mod{\operatorname{mod}}

\def\n{\mathfrak{n}}
\def\p{\mathfrak{p}}
\def\q{\mathfrak{q}}
\def\spec{\operatorname{Spec}}
\def\sing{\operatorname{Sing}}
\def\V{\mathrm{V}}

\def\edd{\operatorname{edd}}
\def\a{\mathfrak{a}}
\def\NG{\operatorname{NG}}
\def\NF{\operatorname{NF}}
\def\RHom{\operatorname{\mathbf{R}Hom}}

\begin{document}
\allowdisplaybreaks
\title{Stability of annihilators of cohomology and closed subsets defined by Jacobian ideals }
\author{Kaito Kimura}
\address{Graduate School of Mathematics, Nagoya University, Furocho, Chikusaku, Nagoya 464-8602, Japan}
\email{m21018b@math.nagoya-u.ac.jp}
\thanks{The author was partly supported by Grant-in-Aid for JSPS Fellows Grant Number 23KJ1117.}

\begin{abstract}
Let $R$ be a commutative Noetherian ring of dimension $d$.
In this paper, we first show that some power of the cohomology annihilator annihilates the $(d+1)$-th Ext modules for all finitely generated modules when either $R$ admits a dualizing complex or $R$ is local.
Next, we study the Jacobian ideal of affine algebras over a field and equicharacteristic complete local rings, and characterize the equidimensionality of the ring in terms of the singular locus and the closed subsets defined by the cohomology annihilator and the Jacobian ideal.
\end{abstract}
\maketitle
%%%%%%%%%%%%%%%%%%%%%%%%%%%%%%%%%%%%%%%%%%%%%%%%%%%%%%%%%%%%
\section{Introduction} 

Throughout the present paper, all rings are assumed to be commutative and Noetherian.
%First, we consider the number at which the cohomology annihilators stabilize up to radical.
For a ring $R$ and an integer $n$, denote by $\ca^n(R)$ the ideal consisting of elements $a$ such that $a\Ext_R^n(M,N)=0$ for all finitely generated $R$-modules $M,N$.
The union $\bigcup_{n\ge 0}\ca^n(R)$ is called \textit{cohomology annihilator} of $R$, which is denoted by $\ca(R)$.
The ascending chain of radicals of $\ca^n(R)$ is stable for all large $n$ since $R$ is Noetherian.
Iyengar and Takahashi \cite{IT} proved that when $R$ is either a localization of an affine algebra over a field or an equicharacteristic excellent local ring, the radical of $\ca^{2d+1}(R)$ is equal to that of $\ca(R)$, where $d=\dim R$.
They also showed that these are defining ideal of the singular locus $\sing R$ of $R$, which is the set of prime ideals $\p$ of $R$ such that $R_\p$ is not regular.
Dey and Takahashi \cite{DeyT} implicitly proved that when $R$ is a Cohen--Macaulay local ring with a canonical module, the radical of $\ca^n(R)$ is stable for all $n\ge \dim R+1$; see Remark \ref{cm case}(1).
The main result in this direction refines the results mentioned above.
 
\begin{thm}[Theorem \ref{V(d+1)} and Corollary \ref{SingV=(d+1)}]\label{main1 V(d+1)}
Let $R$ be a ring of dimension $d$.
Suppose either that $R$ admits a dualizing complex or that $R$ is local.
Then $\V(\ca(R))=\V(\ca^{d+1}(R))$ holds.
In particular, if $R$ is quasi-excellent, $\sing R=\V(\ca^{d+1}(R))$ holds.
\end{thm}

\noindent Theorem \ref{main1 V(d+1)} provides the smallest number at which $\{\sqrt{\smash[b]{{\ca^n(R)}}}\}_{n\ge 0}$ stabilizes under several standard assumptions; see Remark \ref{cm case}(2).
The latter part of Theorem \ref{main1 V(d+1)} is a consequence of results in \cite{DLT, IT}.

%没案
%Next, we study the relationship between  the equidimensionality of rings, the Jacobian ideal and the singular locus.
%Let $R$ be an affine algebra over a field $k$ (that is, a finitely generated $k$-algebra).
%Then $R=S/I$ for some polynomial ring $S=k[X_1,\dots,X_m]$ over $k$ and some ideal $I=(f_1,\ldots,f_n)$ of $S$.
%The $m\times n$ matrix $(\partial f_j/\partial X_i)$ is called the Jacobian matrix.
%We denote by $J_r^k(R)$ the ideal of $R$ generated by the $(r+s)$-minors of the Jacobian matrix, where $s=\height I$.
%特にn=0のときのJ_n^k(R)を（k上の) Jacobian idealと呼ぶ
%For any equicharacteristic complete local ring $R$, there is a formal power series ring $S=k\llbracket X_1, \ldots,X_m\rrbracket$ over a field $k$ and an ideal $I=(f_1,\ldots,f_n)$ of $S$ such that $R=S/I$.
%Both the Jacobian matrix and the ideal $J_r(R)$ are defined similarly.
Let $k$ be a field, and $R=k[X_1,\dots,X_m]/(f_1,\ldots,f_n)$ a quotient of a polynomial ring (resp. $R=k\llbracket X_1,\dots,X_m\rrbracket /(f_1,\ldots,f_n)$ a quotient of a formal power series ring) over $k$.
We denote by $J_r^k(R)$ (resp. $J_r(R)$) the ideal of $R$ generated by the $(r+s)$-minors of the $m\times n$ matrix $(\partial f_j/\partial X_i)$ called the Jacobian matrix, where $s=\height (f_1,\ldots,f_n)$.
The supremum of $\dim R-\dim R/\p$ over all minimal prime ideals $\p$ of $R$ is denoted by $\edd R$.
The ideal $J_0^k(R)$ (resp. $J_0(R)$) is called the \textit{Jacobian ideal}, and $R$ is called \textit{equidimensional} if $\edd R=0$.
The following is a well-known characterization of the regularity of rings via the Jacobian ideal, obtained as a corollary of the classical result known as the Jacobian criterion: for any affine $k$-algebra (resp. an equicharacteristic complete local ring) $R$ with $e=\edd R$, one has $\sing R\subseteq\V(J_e^k(R))$ and $\spec R=\V(J_{e+1}^k(R))$ (resp. $\sing R\subseteq\V(J_{e}(R))$ and $\spec R=\V(J_{e+1}(R))$).
The theorem below refines this fact.

\begin{thm}[Corollary \ref{main cor of edd}]\label{main equidim}
Let $k$ be a perfect field and let $n\ge 0$ be an integer.
\begin{enumerate}[\rm(1)]
\item For an affine $k$-algebra $R$ of dimension $d$, the following are equivalent:
\begin{enumerate}[\rm(i)]
\item The inequality $\edd R\le n$ holds;
\item For any field $l$ and any affine $l$-algebra $S$ with $\spec S\cong\spec R$, $\sing S \subseteq \V(J_n^l(S))$ holds;
\item For any field $l$ and any affine $l$-algebra $S$ with $\spec S\cong\spec R$, $\spec S = \V(J_{n+1}^l(S))$ holds;
\item For any field $l$ and any affine $l$-algebra $S$ with $\spec S\cong\spec R$, $J_n^l(S) \subseteq \sqrt{\ca^{d+1}(S)}$ holds.
\end{enumerate}
\item For an equicharacteristic complete local ring $(R,\m,k)$ of dimension $d$, the following are equivalent:
\begin{enumerate}[\rm(i)]
\item The inequality $\edd R\le n$ holds;
\item For any equicharacteristic complete local ring $S$ with $\spec S\cong\spec R$, $\sing S \subseteq \V(J_n(S))$;
\item For any equicharacteristic complete local ring $S$ with $\spec S\cong\spec R$, $\spec S = \V(J_{n+1}(S))$;
\item For any equicharacteristic complete local ring $S$ with $\spec S\cong\spec R$, $J_n(S) \subseteq \sqrt{\ca^{d+1}(S)}$.
\end{enumerate}
\end{enumerate}
\end{thm}
\noindent In (ii), (iii), and (iv) of the above theorem, it is sufficient to consider only the case where the (residue) field is $k$; see Corollary \ref{main cor of edd}.
As mentioned earlier, the implications (i)$\Rightarrow$(ii) and (i)$\Rightarrow$(iii) are known.
It is worth mentioning that even if $\sing R\subseteq\V(J_n^k(R))$ and $\spec R=\V(J_{n+1}^k(R))$ hold for an affine algebra $R$ over a field $k$, it does not necessarily follow that $\edd R\le n$; see Example \ref{fixringnothold}.
Our new idea is to view $\edd R$ as a topological invariant of the spectrum of the ring and to show (ii)$\Rightarrow$(i) and (iii)$\Rightarrow$(i).
It is clear that (iv)$\Rightarrow$(ii) holds.
Thanks to Theorem \ref{main1 V(d+1)}, the converse (ii)$\Rightarrow$(iv) is obtained.

Consider Theorem \ref{main equidim} in the case when $n=0$.
In this case, the inclusion relations in (ii) and (iv) can be rewritten as equalities; see Corollary \ref{equidimversionofmain}.
Theorem \ref{main equidim} provides a necessary and sufficient condition for a ring to be equidimensional in terms of the Jacobian ideal.
By Wang \cite{W, W3}, under some assumptions, and by Iyengar and Takahashi \cite{IT2}, in the general case, it was proved that $\ca^{d+1}(R)$ contains some power of the Jacobian ideal of $R$ when $R$ is equidimensional, which means that (i)$\Rightarrow$(iv) holds.
Combining this with the trivial implication (iv)$\Rightarrow$(ii), we have (i)$\Rightarrow$(ii) that is a fact derived from the Jacobian criterion mentioned earlier.
Theorem \ref{main1 V(d+1)} says that (i)$\Rightarrow$(ii) and (i)$\Rightarrow$(iv) are equivalent.

Other results of this paper are explained.
In Theorem \ref{main equidim}, we considered all rings for which the spectrum are homeomorphic. 
On the other hand, in Theorems \ref{affine ne} and \ref{ct and t}, we can compare the inclusion relations in (ii) with the equalities in (iii) for a fixed ring.
Using these theorems and the result in \cite{IT2}, Propositions \ref{gen. of EandW} and \ref{new d+1} provide an elementary proof of (i)$\Rightarrow$(ii) and (i)$\Rightarrow$(iii) in Theorem \ref{main equidim} without employing the terminology of smoothness used in the original proof of the Jacobian criterion.

The organization of this paper is as follows. 
In Sections 2, we consider the relationship between the closed subset of the spectrum of a ring defined by the cohomology annihilator and the singular locus, and prove Theorem \ref{main1 V(d+1)}.
In Section 3, we study the ideal generated by the minors of the Jacobian matrix and give Theorem \ref{main equidim}.
Section 4 is an appendix and provides, using elementary arguments, that the Jacobian ideal is well-defined for equicharacteristic complete local rings.
At the end of Section 4, an unresolved question is presented.

%%%%%%%%%%%%%%%%%%%%%%%%%%%%%%%%%%%%%%%%%%%%%%%%%%%%%%%%%%%%
\section{Asymptotic stability of the radicals of cohomology annihilators}

This section studies the number at which the ascending sequence formed by the radicals of the cohomology annihilator stabilizes.
The main result of this section provides the smallest number among such ones under several assumptions.
First of all, we state the definitions of notions used in this paper.

\begin{dfn}
Let $R$ be a ring and $M$ a finitely $R$-module.
We denote by $\mod R$ the category of finitely generated $R$-modules
For every integer $n\ge 0$, we denote by $\ca^n(R)$ the ideal consisting of elements $a$ such that $a\Ext_R^n(M,N)=0$ for all $M,N\in\mod R$. 
The union $\bigcup_{n\ge 0}\ca^n(R)$ is called \textit{cohomology annihilator} of $R$, which is denoted by $\ca(R)$.
The \textit{singular locus} $\sing(R)$ of $R$ is defined as the set of prime ideals $\p$ of $R$ such that $R_\p$ is not a regular local ring and the non-Gorenstein locus $\NG R$ of $R$ is defined as the set of prime ideals $\p$ of $R$ such that $R_\p$ is not Gorenstein.
The non-free locus $\NF_R(M)$ of $M$ is the set of prime ideals of $R$ such that $M_\p$ is not a free $R_\p$-module.
For each ideal $I$ of $R$, the set of prime ideals of $R$ which contain $I$ is denoted by $\V(I)$.
The completion of $R$ is denoted by $\widehat{R}$ when $R$ is local.
\end{dfn}

A similar argument to the proof of \cite[Proposition 2.4(1)]{K} shows the lemma below.

\begin{lem}\label{2.4(1)ofK}
Let $R$ be a local ring.
Then there are equalities $\sqrt{\smash[b]{\ca(R)}}=\sqrt{\smash[b]{\ca(\widehat{R})\cap R}}$ and $\sqrt{\smash[b]{\ca^{n}(R)}}=\sqrt{\smash[b]{\ca^{n}(\widehat{R})\cap R}}$ for all integers $n\ge 0$.
\end{lem}

We prepare a proposition concerning the non-Gorenstein locus, which plays an essential role in the proof of Theorem 2.4.
It is well known that the non-Gorenstein locus is a closed set if the ring has a dualizing complex (or equivalently, if it is a homomorphic image of a finite dimensional Gorenstein ring).
The following proposition characterizes that closed set in terms of the non-free locus of a finitely generated module and an annihilator ideal.

\begin{prop}\label{NGlocus}
Let $R$ be a ring of dimension $d$ with dualizing complex $D=(\cdots\to 0\to D^0\to D^1\to\cdots\to D^d\to 0\to\cdots)$ and let $P=(\cdots\to P^{i-1}\xrightarrow{d^{i-1}}P^i\xrightarrow{d^i}P^{i+1}\to\cdots)$ be a complex of finitely generated projective $R$-modules such that $P^i=0$ for all $i\gg 0$ and $H^j(P)\cong H^j(D)$ for every integer $j$.
Then $\NG R=\NF_R(\cok d^{-1})=\V(\bigcap_{M\in\mod R}\Ann\Ext_R^{d+1}(M,R))$.
\end{prop}

\begin{proof}
We put $C=\cok d^{-1}$ and $I=\bigcap_{M\in\mod R}\Ann\Ext_R^{d+1}(M,R)$.
Let $\p$ be a prime ideal of $R$.
If $\p$ does not contain $I$, then it also does not contain $\Ann\Ext_R^{d+1}(R/\p,R)$.
We see that $\Ext_R^{d+1}(R/\p,R)_\p=0$ and that $R_\p$ is Gorenstein.
We obtain $\NG R\subset \V(I)$.
Suppose that $R_\p$ is Gorenstein.
Then $H^0(P_\p)\cong H^0(D_\p)\cong R_\p$ and $H^j(P_\p)\cong H^j(D_\p)=0$ for any $j\ne 0$.
For any $i\ge 0$, there exists a short exact sequence $0\to \ker d^i_\p \to P^i_\p \to \ker d^{i+1}_\p\to 0$ as $\image d^i_\p\cong\ker d^{i+1}_\p$.
For some integer $n>d$, $P^n=0$ and hence $\ker d^{n-1} \cong P_{n-1}$.
By induction on $i$, $\ker d^i_\p$ is a free $R_\p$-module for all $0<i<n$.
In particular, $\image d^0_\p\cong\ker d^1_\p$ is free.
The natural short exact sequence $0\to H^0(P_\p) \to \cok d^{-1}_\p \to \image d^0_\p\to 0$ says that $C_\p\cong\cok d^{-1}_\p$ is free.
Therefore we get $\NF_R(C)\subset \NG R$.

Finally, we prove $\V(I)\subset\NF_R(C)$. 
Fix $M\in\mod R$.
As $\V(\Ann\Ext_R^1(C,\Omega C))=\NF_R(C)$, we have only to show that $\Ann\Ext_R^1(C,\Omega C)\subset \Ann\Ext_R^{d+1}(M,R)$.
There is a short exact sequence 
\[
  \xymatrix@C=20pt@R=20pt
  {
    G := ( 
    \cdots \ar[r]
    & 0 \ar[r] \ar[d]
    & 0 \ar[r] \ar[d]
    & P^1 \ar[r]^{d^1} \ar@{=}[d]
    & \cdots \ar[r]^{d^i}
    & P^i \ar[r]^{d^{i+1}} \ar@{=}[d] 
    & \cdots )\\
    F = (\cdots \ar[r]
    & P^{-1}\ar[r]^{d^{-1}} \ar@{=}[d]
    & P^0 \ar[r]^{d^0} \ar@{=}[d]
    & P^1 \ar[r]^{d^1} \ar[d]
    & \cdots \ar[r]^{d^i}
    & P^i \ar[r]^{d^{i+1}} \ar[d] 
    & \cdots )\\
    H := (\cdots \ar[r]
    & P^{-1}\ar[r]^{d^{-1}}
    & P^0 \ar[r]^{d^0}
    & 0 \ar[r]
    & \cdots \ar[r]
    & 0 \ar[r]
    & \cdots )\\
  }
\]
of complexes and the natural quasi-isomorphism $H\to C$ since $H^j(P)\cong H^j(D)=0$ for all $j<0$.
Applying the derived functor $\RHom_R(M,\RHom_R(-,D))$ to the exact triangle $G\to F\to C\rightsquigarrow$ in the (bounded) derived category, we have an exact triangle 
$$
\RHom_R(M,\RHom_R(C,D))\to\RHom_R(M,\RHom_R(F,D))\to\RHom_R(M,\RHom_R(G,D))\rightsquigarrow,
$$
which induces the exact sequence 
\begin{equation}\label{d+1 exact}
H^{d+1}(\RHom_R(M,\RHom_R(C,D)))\to \Ext_R^{d+1}(M,R)\to H^{d+1}(\RHom_R(M,\RHom_R(G,D))) \tag{\ref{NGlocus}.1}
\end{equation}
of $R$-modules via the quasi-isomorphism $\RHom_R(F,D)\simeq\RHom_R(D,D)\simeq R$.
Noting that $G$ is a bounded complex of finitely generated projective $R$-modules and $D$ is a bounded complex of injective $R$-modules, $\RHom_R(M,\RHom_R(G,D))$ is quasi-isomorphic to $\Hom_R(M,\Hom_R(G,D))$.
For any integer $i$, if $i\le 0$, then $G^i=0$; otherwise, $D^{i+d+1}=0$.
Hence $\Hom_R(M,\Hom_R(G,D))^{d+1}=\Hom_R(M,\Hom_R(G,D)^{d+1})=0$, which means $H^{d+1}(\RHom_R(M,\RHom_R(G,D)))=0$.
By (\ref{d+1 exact}), we have $\Ann H^{d+1}(\RHom_R(M,\RHom_R(C,D)))\subset\Ann\Ext_R^{d+1}(M,R)$.

Let $a\in\Ann\Ext_R^1(C,\Omega C)$. 
It follows from \cite[Lemma 3.8]{DeyT} that the multiplication by $a$ on $C$ factors through some free module $R^{\oplus m}$.
Applying $H^{d+1}(\RHom_R(M,\RHom_R(-,D)))$, the multiplication by $a$ on $H^{d+1}(\RHom_R(M,\RHom_R(C,D)))$ factors through $H^{d+1}(\RHom_R(M,D))^{\oplus m}$.
Similarly, $\RHom_R(M,D)$ is quasi-isomorphic to $\Hom_R(M,D)$ and $\Hom_R(M,D)^{d+1}=\Hom_R(M,D^{d+1})=0$.
So $a\in\Ann H^{d+1}(\RHom_R(M,\RHom_R(C,D)))\subset\Ann\Ext_R^{d+1}(M,R)$ as $H^{d+1}(\RHom_R(M,D))^{\oplus m}=0$.
\end{proof}

In general, for any complex $X=(\cdots\to X^i\to X^{i+1}\to\cdots)$ of $R$-modules such that $H^i(X)$ are finitely generated over $R$ and $H^j(X)=0$ for all integers $i$ and $j\gg 0$, there is a complex $P$ of finitely generated projective $R$-modules and a quasi-isomorphism $P\to X$ such that ${\rm sup}\{i\mid P^i\ne 0\}={\rm sup}\{i\mid H^i(X)\ne 0\}$; see \cite[Theorem (A.3.2)(L)]{C} for instance.
So, the result below is a direct corollary of Proposition \ref{NGlocus}.

\begin{cor}\label{NGlocus cor}
Let $R$ be a ring of dimension $d$ with dualizing complex.
Then 
$$\NG R=\V\Bigl(\bigcap\nolimits_{M\in\mod R}\Ann\Ext_R^{d+1}(M,R)\Bigr).$$
\end{cor}

Suppose that $R$ is a $d$-dimensional Cohen--Macaulay local ring with canonical module $\omega$.
One has the equality $\NG R=\NF_R(\omega)$.
In this case, Corollary \ref{NGlocus cor} follows immediately from \cite[Theorem 2.3]{DKT} because it says that $\Ann\Ext_R^1(\omega,\Omega \omega)=\bigcap_{M\in\mod R}\Ann\Ext_R^{d+1}(M,R)$ holds.
On the other hand, Proposition \ref{NGlocus} can also be viewed as a non-Cohen--Macaulay version of \cite[Theorem 2.3]{DKT}.
Indeed, with the notation of the proof of Proposition \ref{NGlocus}, when $R$ is Cohen--Macaulay, $\image d^i\cong\ker d^{i+1}$ are projective for all $i\ge 0$ and up to a projective summand $\image d^0$, $C$ is isomorphic to $H^0(P)$, which is a canonical module of $R$.

The main result of this section is the following theorem.
Let $R$ be a finite dimensional ring and let $n>\dim R$ be an integer.
Theorem \ref{V(d+1)} asserts that the radical of $\ca(R)$ is equal that of $\ca^n(R)$ if either $R$ admits a dualizing complex or $R$ is local.

\begin{thm}\label{V(d+1)}
Let $R$ be a ring of dimension $d$.
Suppose either that $R$ admits a dualizing complex or that $R$ is local.
Then $\V(\ca(R))=\V(\ca^{d+1}(R))$ holds.
\end{thm}

\begin{proof}
First, we deal with the case where $R$ admits a dualizing complex.
Since $\ca^{d+1}(R)\subset\ca(R)$, it suffices to prove that $\ca(R)\subset\sqrt{\ca^{d+1}(R)}$.
Put $I=\bigcap_{M\in\mod R}\Ann\Ext_R^{d+1}(M,R)$ and take $n\ge d$ such that $\ca^{n+1}(R)=\ca(R)$.
It follows from Corollary \ref{NGlocus cor} and \cite[Lemma 2.10(2)]{IT} that $\V(I)=\NG R\subset\sing R\subset\V(\ca(R))$. 
Thus $I$ contains $\ca(R)^l$ for some $l>0$.
Let $M,N$ be finitely generated $R$-modules.
A short exact sequence $0\to \Omega N\to R^{\oplus m}\to N\to 0$ induces a long exact sequence 
$$
\cdots \to \Ext_R^i(M,R^{\oplus m})\to\Ext_R^i(M,N)\to \Ext_R^{i+1}(M,\Omega N) \to \cdots.
$$
This means that $I\cdot\ca^{i+1}(R)\subset\ca^i(R)$ for every $i>d$.
We obtain $I^{n-d}\cdot\ca^{n+1}(R)\subset\ca^{d+1}(R)$ and hence $\ca(R)^{l(n-d)+1}\subset\ca^{d+1}(R)$.

Next, We handle the case where $R$ is local.
The equality $\sqrt{\smash[b]{\ca(\widehat{R})}}=\sqrt{\smash[b]{\ca^{d+1}(\widehat{R})}}$ holds since $\widehat{R}$ admits a dualizing complex.
The assertion follows from Lemma \ref{2.4(1)ofK}.
\end{proof}

The cohomology annihilator is defining ideal of the singular locus under several assumptions.
Corollary \ref{SingV=(d+1)} plays an important role in proving Theorem \ref{main equidim}, which is one of the main results of this paper.

\begin{cor}\label{SingV=(d+1)}
Let $R$ be a quasi-excellent ring of dimension $d$.
Suppose either that $R$ admits a dualizing complex or that $R$ is local.
Then $\sing R=\V(\ca^{d+1}(R))$ holds.
\end{cor}

\begin{proof}
It follows from \cite[Corollary C]{DLT} and \cite[Theorem 1.1]{IT} that $\sing R=\V(\ca(R))$ holds.
One has the equality $\sing R=\V(\ca^{d+1}(R))$ by Theorem \ref{V(d+1)}.
\end{proof}

Theorem \ref{V(d+1)} and Corollary \ref{SingV=(d+1)} improve and recover several existing results.

\begin{rem}\label{cm case}
\noindent {\rm (1)} 
When $R$ is a Cohen--Macaulay local ring with a canonical module $\omega$, Theorem \ref{V(d+1)} can be proved more easily by using existing results instead of Corollary \ref{NGlocus cor}.
Indeed, we put $\ca^{n+1}(R)=\ca(R)$ for some $n\ge d$.
It follows from \cite[Propositon 4.2(2)]{DeyT} (or \cite[Theorem 2.3]{DKT}) that 
$$
(\operatorname{tr}\omega )^n\cdot \bigcap_{X\in\cm R, Y\in\mod R} \Ann\Ext_R^{n+1}(X,\Omega^n Y) \subset \bigcap_{X\in\cm R, Y\in\mod R} \Ann\Ext_R^1(X,Y)
$$
where $\cm R$ is the subcategory of $\mod R$ consisting of maximal Cohen--Macaulay $R$-modules and $\operatorname{tr}\omega$ is the trace ideal of $\omega$; see \cite{DKT, DeyT} for instance.
As $\V(\operatorname{tr}\omega)=\NG R$, similar to the proof of Theorem \ref{V(d+1)}, $\ca(R)^l\subset \operatorname{tr}\omega$ for some $l>0$.
The relations 
\begin{align*}
\ca(R)&=\ca^{n+1}(R)\subset\bigcap_{X\in\cm R, Y\in\mod R} \Ann\Ext_R^{n+1}(X,\Omega^n Y) \quad {\rm and} \\
\bigcap_{X\in\cm R, Y\in\mod R} &\Ann\Ext_R^1(X,Y) \subset \bigcap_{X,Y\in\mod R} \Ann\Ext_R^1(\Omega^d X,Y)=\ca^{d+1}(R)
\end{align*}
deduce $\ca(R)^{ln+1}\subset\ca^{d+1}(R)$, which means $\V(\ca(R))=\V(\ca^{d+1}(R))$.

\noindent {\rm (2)} Let $R$ be as in Theorem \ref{V(d+1)}.
Note that $\V(\ca^d(R))$ does not necessarily equal to $\V(\ca^{d+1}(R))$.
Indeed, if $R$ is either a polynomial ring or a formal power series ring over a field, then $\ca^d(R)=0$ and $\ca^{d+1}(R)=R$.
Moreover, for any ring $R$ of dimension $d$, we see that $\ca^d(R)$ is contained in any prime ideal $\p$ of $R$ such that $\dim R/\p=d$.
Now we prove this claim.
Take a maximal ideal $\m$ of $R$ such that it contains $\p$ and $\height\m=d$.
Since $\ca^d(R)_\m\subseteq\ca^d(R_\m)$, we may assume that $R$ is local.
There is a prime ideal $\q$ of $\widehat{R}$ such that $\dim \widehat{R}/\q=d$ and $\p=\q\cap R$ as $\dim \widehat{R}/\p \widehat{R}=\dim R/\p=d$.
By lemma \ref{2.4(1)ofK}, $\ca^{d}(\widehat{R})\subseteq\q$ implies $\ca^{d}(R)\subseteq\p$ and thus we may assume that $R$ is complete.
There is a Gorenstein local ring $S$ of dimension $d$ such that $R$ is a homomorphic image of $S$.
The inverse image of $\p$ is a minimal prime ideal of $S$ because $\dim R/\p=d=\dim S$.
We obtain $\Hom_S(R,S)_\p\ne 0$ and hence $\Ann_R\Hom_S(R,S)\subseteq\p$.
An analogous argument to the proof of \cite[Proposition 2.6]{K} shows that $\ca^{d}(R)$ is contained in $\Ann_R H_\m^d(R)=\Ann_R \Hom_S(R,S)$.
The proof of the claim is now completed.
In particular, if $R$ is as in Corollary \ref{SingV=(d+1)} and it is reduced, then we have $\sing R=\V(\ca^{d+1}(R))\subsetneq\V(\ca^d(R))$.

\noindent {\rm (3)}  Corollary \ref{SingV=(d+1)} improves \cite[Theorems 5.3 and 5.4]{IT}, that is to say, Corollary \ref{SingV=(d+1)} replaces $2d+1$ with $d+1$.
Indeed, if $R$ is an equicharacteristic excellent local ring of dimension $d$, the equality $\sing R=\V(\ca^{d+1}(R))$ immediately follows from \cite[Theorem 5.3]{IT} and Corollary \ref{SingV=(d+1)} in the case where $R$ is local.
On the other hand, If $R$ is a localization of an affine algebra over a field, it is excellent and a homomorphic image of a finite-dimensional Gorenstein ring.
So, the same equality is a consequence of \cite[Theorem 5.4]{IT} and Corollary \ref{SingV=(d+1)} in the case where the ring admits a dualizing complex.

\noindent {\rm (4)} Suppose that $R$ is either an affine algebra over a field or an equicharacteristic complete local ring, of dimension $d$.
It is well-known by the Jacobian criterion that $\V(J)$ contains $\sing R$ when $R$ is equidimensional, where $J$ is the Jacobian ideal of $R$; see Definition \ref{jac ideal dfn}.
By this and Corollary \ref{SingV=(d+1)}, we have $\V(\ca^{d+1}(R))\subseteq\V(J)$.
Corollary \ref{SingV=(d+1)} deduces a bit weaker version of \cite[Theorem 1.1]{IT2}, which asserts that some power of $J$ annihilates $\Ext_R^{d+1}(M,N)$ for all $R$-modules $M,N$, not necessarily finitely generated.
As mentioned in \cite{L}, it is necessary to assume in \cite[Theorem 1.1]{IT2} that R is equidimensional.
%\footnote{There is a typo in \cite[Theorem 1.1]{IT2}: it is necessary to assume $R$ is equidimensional.}, %\old{}
\end{rem}

We close this section by providing an example of a ring $R$ such that $\ca^{d+1}(R)$ is the defining ideal of the singular locus of $R$, where $d=\dim R$.

\begin{ex}
Let $R=k \llbracket X, Y \rrbracket/(X^2, XY)$ be a quotient of a formal power series ring over a field $k$.
The ring $R$ is a complete local ring of dimension 1 that has an isolated singularity.
According to Corollary \ref{SingV=(d+1)}, the equalities $\V((X,Y)R)=\sing R=\V(\ca^2(R))$ should hold.
The equality $\ca^3(R)=(X,Y)R$ was proven in \cite[Example 5.7]{L}.
We see that $\ca^2(R)=(X,Y)R$ by using the method employed there.
Since $XR$ is contained in the socle of $R$, it is also contained in $\ca^1(R)$ by \cite[Example 2.6]{IT}.
We have only to show $YR\subseteq\ca^2(R)$.
For any $m\ge 2$, $R/(X,Y^m)R$ is not a submodule of any free $R$-module.
In fact, if a submodule of a free $R$-module is annihilated by $(X,Y^m)R$, then it is also annihilated by $(X,Y)R$.
Let $L, M\in\mod R$ and let $N$ be a first syzygy of $M$, which is minimal.
We get $XN=0$ as $XR$ is contained in the socle of $R$ and $N$ is minimal.
This means that $N$ is a finitely generated module over $R/XR\cong k \llbracket Y \rrbracket$, which is PID.
Hence $N$ is a finite direct sum of $R/XR$ and $k$ because $R/(X,Y^m)R$ is not a submodule of any free $R$-module for any $m\ge 2$.
We see that $Y\Ext_R^1(k,L)=0$ and the exact sequence $0\to k\cong XR\to R\to R/XR\to 0$ induces an exact sequence  $\Hom_R(k,L)\to \Ext_R^1(R/XR,L)\to \Ext_R^1(R,L)=0$, which implies $Y\Ext_R^1(R/XR,L)=0$.
We obtain $Y\Ext_R^2(M,L)=Y\Ext_R^1(N,L)=0$ and conclude that $YR$ is contained in $\ca^2(R)$.
\end{ex}

%%%%%%%%%%%%%%%%%%%%%%%%%%%%%%%%%%%%%%%%%%%%%%%%%%%%%%%%%%%%
\section{Jacobian ideals and singular loci}

In this section, we study the relationship between Jacobian ideals and singular loci for affine algebras over a field and equicharacteristic complete local rings, and characterize the equidimensionality of these rings in those terms.
For the time being, we will focus on defining the notations used in the following sections and providing remarks about them.

\begin{dfn}
Let $R$ be a ring, and $m, n\ge 0$. 

{\rm (1)} Let $A$ be an $m\times n$ matrix over $R$.
We denote by $I_r(A)$ the ideal generated by the $r$-minors of $A$ (i.e., the determinants of the $r \times r$ submatrices) for $1\le r\le {\rm min}\{m,n\}$.
We also set $I_r(A)=0$ for $r>{\rm min}\{m,n\}$ and $I_r(A)=R$ for $r\le 0$.

{\rm (2)} Let $S$ be either the polynomial ring $R[X_1, \ldots,X_m]$ or the formal power series ring $R\llbracket X_1, \ldots,X_m\rrbracket$.
For $f\in S$ and $1\le i\le m$, the partial derivative of $f$ with respect to $X_i$ is denoted by $\partial f/\partial X_i$ (or $\frac{\partial f}{\partial X_i}$).
That is, if $f=\sum_{k=(k_1,\cdots,k_m)} a_k X_1^{k_1}\cdots X_m^{k_m} \in S$, then $\partial f/\partial X_i=\sum_{k} k_i a_k X_1^{k_1}\cdots X_i^{k_i-1}\cdots X_m^{k_m}$, where $a_k\in R$.
For $f_1,\ldots,f_n\in S$, the $m\times n$ matrix
\begin{align*}
\left(\frac{\partial f_j}{\partial X_i}\right)=
\begin{pmatrix}
   \dfrac{\partial f_1}{\partial X_1} & \cdots &  \dfrac{\partial f_n}{\partial X_1} \\
   \vdots & \ddots & \vdots \\
   \dfrac{\partial f_1}{\partial X_m} & \cdots  &  \dfrac{\partial f_n}{\partial X_m} \\
\end{pmatrix}
\end{align*}
is called the \textit{Jacobian matrix} of $f_1,\ldots,f_n$.

{\rm (3)} Let $k$ be a field. 
For an affine $k$-algebra $R$, that is $R$ is a finitely generated algebra over $k$, we define $\codim_{k}(R)={\rm inf}\{m-\dim R \mid$ there is a surjective $k$-algebra homomorphism from $k[X_1, \ldots,X_m]$ to $R\}$.
For an equicharacteristic complete local ring $R$ with residue field $k$, we also define $\codim(R)={\rm inf}\{m-\dim R \mid$ there is a surjective ring homomorphism from $k\llbracket X_1, \ldots,X_m\rrbracket$ to $R\}$.

{\rm (4)} When $R$ is a finite dimensional ring, we denote by $\edd R$ the supremum of $\dim R-\dim R/\p$ over all minimal prime ideals $\p$ of $R$.
(Note that $R$ is equidimensional if and only if $\edd R=0$.)
\end{dfn}

\begin{rem}\label{homeo} 
{\rm (1)} For an affine algebra $R$, $\codim_{k}(R)$ depends on a field $k$. 
Indeed, $\codim_{\mathbb{R}}(\mathbb{C})=1$ and $\codim_{\mathbb{C}}(\mathbb{C})=0$, where $\mathbb{R}$ is the field of real numbers and $\mathbb{C}$ is the field of complex numbers.
On the other hand, when an equicharacteristic complete local ring $R$ is given, $\codim R$ is uniquely determined, and $\codim R+\dim R$ coincides with the embedding dimension of $R$.

{\rm (2)} Let $R$ be a finite dimensional ring, and let $V_1,\ldots,V_n$ be all the maximal irreducible closed subsets of $\spec R$.
Put $v_i={\rm sup}\{m\ge 0\mid$ there exists a chain $V_i=W_0 \supsetneq \cdots \supsetneq W_m$ of irreducible closed subsets of $\spec R\}$.
By definition, we see that $\edd R={\rm sup}\{v_i\mid 1\le i\le n\}-{\rm inf}\{v_i\mid 1\le i\le n\}$.
So, $\edd R$ is characterized in terms of the topology of $\spec R$.
This means that for any ring $S$ such that $\spec S$ is homeomorphic to $\spec R$, one has the equality $\edd S=\edd R$.
\end{rem}

\begin{dfn}\label{jac ideal dfn}
Let $k$ be a field. 

{\rm (1)} Let $R$ be an affine $k$-algebra.
Suppose that $\phi:S=k[X_1, \ldots,X_m] \to R$ is a surjective $k$-algebra homomorphism.
We set $\ker\phi=(f_1,\ldots,f_r)$ and $d=m-\dim R$, where $f_1,\ldots,f_r\in S$.
For each integer $n$, the ideal $\phi(I_{n+d}(\partial f_j/\partial X_i))$ is uniquely determined regardless of the choices of $S$, $\phi$, and $f_1,\ldots,f_r$; see Remark \ref{J is w.d.}.
We denote this ideal by $J_n^k(R)$.
In particular, $J_0^k(R)$ is called the \textit{Jacobian ideal} of $R$ (over $k$) and is denoted by $\jac_k R$.

{\rm (2)} Let $R$ be an equicharacteristic complete local ring with residue field $k$.
Suppose that $\phi:S=k\llbracket X_1, \ldots,X_m\rrbracket \to R$ is a surjective ring homomorphism.
We set $\ker\phi=(f_1,\ldots,f_r)$ and $d=m-\dim R$, where $f_1,\ldots,f_r\in S$.
For each integer $n$, the ideal $\phi(I_{n+d}(\partial f_j/\partial X_i))$ is uniquely determined regardless of the choices of $S$, $\phi$, and $f_1,\ldots,f_r$; see Remark \ref{J is w.d.} and Proposition \ref{local unique}.
Similarly, we denote this ideal by $J_n(R)$.
In particular, $J_0(R)$ is called the \textit{Jacobian ideal} of $R$ and is denoted by $\jac R$.
\end{dfn}

\begin{rem}\label{J is w.d.}
Let $k$ be a field. 

{\rm (1)} Let $R=k[X_1, \ldots,X_m]/(f_1,\ldots,f_r)$ and $S=k[Y_1, \ldots,Y_n]/(g_1,\ldots,g_s)$.
Suppose that there is an isomorphism $\phi:R \xrightarrow{\cong} S$ of $k$-algebras.
Then $\phi$ induces the natural isomorphism $\Omega_{R/k} \xrightarrow{\cong} \Omega_{S/k}$ of modules over $R$ and $S$, where $\Omega_{R/k}$ and $\Omega_{S/k}$ are the module of K\"{a}hler differentials of $R$ and $S$ over $k$, respectively.
There exist exact sequences
$$
R^{\oplus r} \xrightarrow{\bigl(\tfrac{\partial f_j}{\partial X_i}\bigr)} R^{\oplus m} \to \Omega_{R/k} \to 0 \ {\rm and} \ 
S^{\oplus s} \xrightarrow{\bigl(\tfrac{\partial g_j}{\partial Y_i}\bigr)} S^{\oplus n} \to \Omega_{S/k} \to 0;
$$
see \cite[Section 16]{E} and \cite[Section 25]{Mat} for instance.
%(The symbol $(-)^\top$ denotes transpose.)
We have $\phi(I_{m-l}(\partial f_j/\partial X_i))=I_{n-l}(\partial g_j/\partial Y_i)$ for any integer $l$ since these ideals are $l$-th Fitting invariants.
This means that for an affine $k$-algebra $T$ and an integer $u$, the ideal $J_u^k(T)$ of $T$ is well-defined.
The ideal $J_u(T)$ is also well-defined for an equicharacteristic complete local ring $T$ with residue field $k$.
A proof of this fact, which does not use K\"{a}hler differentials, is provided in Section 4.

{\rm (2)} Let $(R,\m)$ be a local affine $k$-algebra with residue field $k$.
Then $R$ is Artinian since the maximal ideal is nilpotent by \cite[Theorem 5.5]{Mat}.
Since $R=k+\m$, we can choose a surjective homomorphism $\phi:k[X_1, \ldots,X_m] \to R$ of $k$-algebras such that $m=\codim_k(R)+\dim R$ and $\phi(X_1, \ldots,X_m)\subseteq \m$.
Considering the $(X_1, \ldots,X_m)$-adic completion, we see that $\codim(R)\le \codim_k(R)$.
On the other hand, let  $\phi:k\llbracket X_1, \ldots,X_m\rrbracket \to R$ be a surjective ring homomorphism such that $m=\codim(R)+\dim R$.
As $R$ is Artinian and $m$ is the embedding dimension of $R$, $(X_1,\ldots, X_m)^n\subseteq\ker\phi\subseteq (X_1,\ldots, X_m)^2$ for some $n\ge 2$.
The natural surjection from $k[X_1, \ldots,X_m]/(X_1,\ldots, X_m)^n=k\llbracket X_1, \ldots,X_m\rrbracket/(X_1,\ldots, X_m)^n$ to $R$ is induced by $\phi$, which means that $k[X_1, \ldots,X_m]/(f_1,\ldots, f_r)=k\llbracket X_1, \ldots,X_m\rrbracket/(f_1,\ldots, f_r) \cong R$ for some $f_1,\ldots, f_r\in (X_1,\ldots, X_m)^2\subseteq k[X_1, \ldots,X_m]$.
This implies that the equalities $\codim_{k}(R)=\codim(R)$ and $J_u^k(R)=J_u(R)$ hold for any integer $u$.

{\rm (3)} Let $R$ be a ring, and $A$ an $m\times n$ matrix over $R$.
We see that $I_{r+1}(A)\subseteq I_r(A)$ by the Laplace expansion.
Hence we have $J_{r+1}^k(R)\subseteq J_r^k(R)$ and $J_{r+1}(S)\subseteq J_r(S)$ for any affine $k$-algebra $R$, any an equicharacteristic complete local ring $S$ with residue field $k$, and any integer $r$.
\end{rem}

Recall well-known facts about the tensor product of quotient rings of polynomial rings and the completed tensor product of quotient rings of power series rings.
The equalities regarding Krull dimensions of rings are obtained by considering Noether normalization and the dimension formula for flat extensions of local rings, respectively; see \cite[Theorems A.11. and A.14.]{BH} for instance.

\begin{lem}\label{t and ct}
Let $k$ be a field.
\begin{enumerate}[\rm(1)]
\item Let $R=k[X_1, \ldots,X_m]/(f_1,\ldots,f_r)$ and $S=k[Y_1, \ldots,Y_n]/(g_1,\ldots,g_s)$.
Then the tensor product $R\otimes_k S$ of $R$ and $S$ over $k$ is isomorphic to $k[X_1, \ldots,X_m, Y_1, \ldots,Y_n]/(f_1,\ldots,f_r, g_1,\ldots,g_s)$ as $k$-algebras in a natural way.
Also, the natural ring homomorphisms $R\to R\otimes_k S$ and $S\to R\otimes_k S$ are flat, and the equality $\dim R\otimes_k S=\dim R+\dim S$ holds.
\item Let $R=k\llbracket X_1, \ldots,X_m\rrbracket/(f_1,\ldots,f_r)$ and $S=k\llbracket Y_1, \ldots,Y_n\rrbracket/(g_1,\ldots,g_s)$.
Then the completed tensor product $R\ \widehat{\otimes}_k\  S$ of $R$ and $S$ over $k$ is isomorphic to $k\llbracket X_1, \ldots,X_m, Y_1, \ldots,Y_n\rrbracket/(f_1,\ldots,f_r, g_1,\ldots,g_s)$ in a natural way.
Also, the natural ring homomorphisms $R\to R\ \widehat{\otimes}_k\  S$ and $S\to R\ \widehat{\otimes}_k\  S$ are flat and local, and the equality $\dim R\ \widehat{\otimes}_k\  S=\dim R+\dim S$ holds.
\end{enumerate}
\end{lem}

Let $R$ be an affine algebra over a field $k$, and let $n$ be an integer.
If $\edd R\le n$, then the conditions $\sing R\subseteq \V(J_n^k(R))$ and $\spec(R)=\V(J_{n+1}^k(R))$ hold; see Section 1 or Propositions \ref{gen. of EandW} and \ref{new d+1}.
Theorem \ref{affine ne} states that these conditions are deeply interconnected.
In fact, for example, we put $A=k[x]/(x^2)$ and $R'=A\otimes_k R$.
Then we see that there is an equality $\edd R=\edd R'$ and that $\sing R'\subseteq \V(J_n^k(R'))$ holds if and only if $\spec(R)=\V(J_{n+1}^k(R))$.
Note that $\sing (R)\subseteq \V(J_n^k(R))$ is neither a necessary nor a sufficient condition for $\spec(R)=\V(J_{n+1}^k(R))$ to hold.; see Example \ref{nhensu example}.

\begin{thm}\label{affine ne}
Let $k$ be a field, $R$ an affine $k$-algebra.
For any integer $n>0$, the following are equivalent:
\begin{enumerate}[\rm(1)]
\item Suppose that $A\to B$ is a flat homomorphism of affine $k$-algebras such that $a=\codim_k A\le n$.
If there is $\m\in\sing A$ such that $B/\m B\cong R$ as $k$-algebras and $\height \m=\dim A$, then $\sing B\nsubseteq \V(J_{n-a}^k(B))$;
\item Suppose that $A\to B$ is a flat homomorphism of affine $k$-algebras such that $\codim_k A\le n$.
If there is $\m\in\sing A$ such that $B/\m B\cong R$ as $k$-algebras and $\height \m=\dim A$, then $\sing B\nsubseteq \V(\jac_k(B))$;
\item Let $A$ be an affine $k$-algebra such that $\codim_k A\le n$.
If there is $\m\in\sing A$ such that $A/\m\cong k$ and $\height \m=\dim A$, then $\sing (A\otimes_k R) \nsubseteq \V(\jac_k(A\otimes_k R))$;
\item $\sing (A\otimes_k R) \nsubseteq \V(\jac_k(A\otimes_k R))$ for some local affine $k$-algebra $A$ with residue field $k$ such that $\codim_k A$ is equal to $n$;
\item $\sing (A\otimes_k R) \nsubseteq \V(J_{n-a}^k(A\otimes_k R))$ for some local affine $k$-algebra $A$ with residue field $k$ such that $a=\codim_k A\le n$;
\item $\spec(R)\ne \V(J_n^k(R))$ holds.
\end{enumerate}
\end{thm}

\begin{proof}
The implication (1)$\Rightarrow$(2) holds since $J_{n-a}^k(B)\subseteq J_0^k(B)=\jac_k(B)$ by Remark \ref{J is w.d.}.
Also, (2)$\Rightarrow$(3) immediately follows from Lemma \ref{t and ct}.
Setting $A=k[X_1, \ldots,X_n]/(X_1, \ldots,X_n)^2$, it follows that (3)$\Rightarrow$(4). 
It is clear that (4)$\Rightarrow$(5) holds.

We prove (5)$\Rightarrow$(6).
Suppose that $\sing (A\otimes_k R) \nsubseteq \V(J_{n-a}^k(A\otimes_k R))$ for some a local affine $k$-algebra with residue field $k$ such that $a=\codim_k A\le n$.
Thanks to Remark \ref{J is w.d.}(2), $A$ is Artinian and we can put $A=k[X_1, \ldots,X_a]/(f_1,\ldots, f_r)$, for some $f_1,\ldots, f_r\in (X_1,\ldots, X_a)^2$.
We set $R=k[Y_1, \ldots,Y_m]/(g_1,\ldots, g_s)$.
Then $B:=k[X_1,\ldots, X_a,Y_1, \ldots,Y_m]/(f_1,\ldots, f_r,g_1,\ldots, g_s)\cong A\otimes_k R$ and $\dim B=\dim R$ by Lemma \ref{t and ct}.
Since $f_1,\ldots, f_r\in k[X_1, \ldots,X_a]$ and $g_1,\ldots, g_s\in k[Y_1, \ldots,Y_m]$, $J_{n-a}^k(B)$ is generated by the $(n+m-\dim R)$-minors of the $(a+m)\times (r+s)$ matrix
\begin{align*}
\begin{pmatrix}
   \Bigl(\dfrac{\partial f_j}{\partial X_i}\Bigr) & 0 \\
   0 &  \Bigl(\dfrac{\partial g_j}{\partial Y_i}\Bigr) \\
\end{pmatrix}
\end{align*}
over $B$.
For any $1\le j\le r$, $f_j$ belongs to $(X_1,\ldots, X_a)^2$, and thus $\partial f_j/\partial X_i$ are in $(X_1,\ldots, X_a)$ for all $1\le i\le a$.
We obtain $J_{n-a}^k(B)\subseteq I_{n+m-\dim R}(\partial g_j/\partial Y_i)B+(X_1,\ldots, X_a)B=J_n^k(R)B+(X_1,\ldots, X_a)B$.
Let $\p\in\sing B \setminus \V(J_{n-a}^k(B))$.
Since $A$ is Artinian, we get $(X_1,\ldots, X_a)B\subseteq\p$, which means $J_n^k(R)B\nsubseteq\p$.
Hence $\p\cap R\notin \V(J_n^k(R))$.

In order to show (6)$\Rightarrow$(1), we assume $\spec(R)\ne \V(J_n^k(R))$.
Set $A=k[X_1, \ldots,X_l]/(f_1,\ldots, f_r)$, where $l=a+\dim A$.
We can write $B=k[X_1, \ldots,X_l,Y_1, \ldots,Y_m]/(f_1,\ldots, f_r, g_1,\ldots, g_s)$ for some polynomials $g_1,\ldots, g_s \in k[X_1, \ldots,X_l,Y_1, \ldots,Y_m]$.
Since $\m\in\spec A$, we can choose polynomials $h_1,\ldots, h_t\in k[X_1, \ldots,X_l]$ such that $\m=(h_1,\ldots, h_t)A$ and $(f_1,\ldots, f_r)\subseteq (h_1,\ldots, h_t)$.
Then we see that $R\cong B/\m B=k[X_1, \ldots,X_l,Y_1, \ldots,Y_m]/(h_1,\ldots, h_t, g_1,\ldots, g_s)$.
Put $e=l+m-\dim R$.
For
\begin{align*}
U=
\begin{pmatrix}
   \Bigl(\dfrac{\partial h_j}{\partial X_i}\Bigr) & \Bigl(\dfrac{\partial g_j}{\partial X_i}\Bigr) \\
   0 & \Bigl(\dfrac{\partial g_j}{\partial Y_i}\Bigr) \\
\end{pmatrix},
\end{align*}
it is seen that $J_n^k(R)=I_{n+e}(U)R$.
Note that $(\partial h_j/\partial X_i)$ is an $l \times t$ matrix.
The $(n+e)$-minors of $U$ are of the form
\begin{align*}
T=
\begin{pmatrix}
   V & \ast \\
   0 & W \\
\end{pmatrix},
\end{align*}
where $V$ is a $p\times q$ submatrix of $(\partial h_j/\partial X_i)$, $W$ is an $(n+e-p)\times (n+e-q)$ submatrix of $(\partial g_j/\partial Y_i)$, $1\le p\le l$ and $1\le q\le t$.
By the Laplace expansion, $\operatorname{det}T=0$ if $p<q$, and $\operatorname{det}T\in I_{n+e-p}(\partial g_j/\partial Y_i)$ if $p\ge q$.
Therefore $J_n^k(R)$ is contained in $I_{n+e-l}(\partial g_j/\partial Y_i)R$.
By $\spec(R)\ne \V(J_n^k(R))$, $J_n^k(R)$ is nonzero, and hence $n+e-l\le {\rm min}\{m, s\}$.

Let $\p$ be a prime ideal of $R$ which does not contain $J_n^k(R)$, and let $\q$ be a prime ideal of $B$ such that $\p=\q/\m B$.
The ideal $J_n^k(R)$ is not contained in $\p$, and neither is $I_{n+e-l}(\partial g_j/\partial Y_i)R$, which means that $\q$ does not contain $I_{n+e-l}(\partial g_j/\partial Y_i)B$.
As $\q\cap A$ of $A$ contains the maximal ideal $\m$, $\q\cap A=\m$ and the natural homomorphism $A_\m\to B_\q$ is flat and local.
We get $\q\in\sing B$ because $\m\in\sing A$.

We see that $\dim B\ge \dim A+\dim R$.
In fact, let $Q$ be a prime ideal of $R$ such that $\height Q=\dim R$, and let $P$ be a prime ideal of $B$ such that $Q=P/\m B$.
Similarly, $P\cap A=\m$.
Since $A\to B$ is flat, we have $\dim B\ge \height P= \height \m +\height P/\m B=\dim A+\dim R$.
We obtain $d:=l+m-\dim B\le (l-\dim A)+(m-\dim R)=a+e-l$, and hence $d+n-a\le n+e-l\le {\rm min}\{m,s\}$.
The ideal $J_{n-a}^k(B)$ is generated by the $(d+n-a)$-minors of the $(n+m)\times (r+s)$ matrix
\begin{align*}
\begin{pmatrix}
   \Bigl(\dfrac{\partial f_j}{\partial X_i}\Bigr) & \Bigl(\dfrac{\partial g_j}{\partial X_i}\Bigr) \\
   0 &  \Bigl(\dfrac{\partial g_j}{\partial Y_i}\Bigr) \\
\end{pmatrix}
\end{align*}
over $B$.
Therefore $I_{n+e-l}(\partial g_j/\partial Y_i)B\subseteq I_{d+n-a}(\partial g_j/\partial Y_i)B\subseteq J_{n-a}^k(B)$, which means that $\q$ does not contain $J_{n-a}^k(B)$ but belongs to $\sing B$.
\end{proof}

The following theorem is an equicharacteristic complete local version of Theorem \ref{affine ne}.
The proof is omitted since it is similar to (or simpler than) that of the above theorem.

\begin{thm}\label{ct and t}
Let $k$ be a field, $R$ an equicharacteristic complete local ring with residue field $k$.
For any integer $n>0$, the following are equivalent:
\begin{enumerate}[\rm(1)]
\item Suppose that $(A,\m,k)\to (B,\n,k)$ is a flat and local homomorphism of equicharacteristic complete local rings such that $B/\m B\cong R$.
If $0<a=\codim A\le n$, then $\sing B\nsubseteq \V(J_{n-a}(B))$;
\item Suppose that $(A,\m,k)\to (B,\n,k)$ is a flat and local homomorphism of equicharacteristic complete local rings such that $B/\m B\cong R$.
If $0<\codim A\le n$, then $\sing B\nsubseteq \V(\jac(B))$;
\item Let $A$ be an equicharacteristic complete local ring with residue field $k$ such that $0<\codim A\le n$.
Then $\sing (A  \widehat{\otimes}_k  R) \nsubseteq \V(\jac(A  \widehat{\otimes}_k  R))$;
\item $\sing (A  \widehat{\otimes}_k  R) \nsubseteq \V(\jac(A  \widehat{\otimes}_k  R))$ for some Artinian equicharacteristic local ring $A$ with residue field $k$ such that $\codim A$ is equal to $n$;
\item $\sing (A  \widehat{\otimes}_k  R) \nsubseteq \V(J_{n-a}(A  \widehat{\otimes}_k  R))$ for some Artinian equicharacteristic local ring $A$ with residue field $k$ such that $a=\codim A\le n$;
\item $\spec(R)\ne \V(J_n(R))$ holds.
\end{enumerate}
\end{thm}

\begin{ex}\label{nhensu example}
Let $R=k[X,Y_1,\ldots,Y_{n+1}]/(XY_1,\ldots,XY_{n+1})$ be a quotient of a polynomial ring over a field $k$, where $n\ge 1$.
The Jacobian matrix of $XY_1,\ldots,XY_{n+1}$ is the $(n+1)\times (n+2)$ matrix
\begin{align*}
T=
\begin{pmatrix}
   Y_1 & X & 0 & \cdots & 0 \\
   Y_2 & 0 & X & \ddots & \vdots \\
   \vdots & \vdots & \ddots & \ddots & 0 \\
   Y_{n+1} & 0 & \cdots & 0 & X \\
\end{pmatrix}.
\end{align*}
Since $\codim_k R=1$, we obtain the equalities $J_{n-1}^k(R)=I_{n}(T)R=(X^n,X^{n-1}Y_1,\ldots,X^{n-1}Y_{n+1})R$ and $J_n^k(R)=I_{n+1}(T)R=(X^{n+1},X^{n}Y_1,\ldots,X^{n}Y_{n+1})R$.
Hence one has $\sing R=\V((X,Y_1,\ldots,Y_{n+1})R)\subseteq\V(J_{n-1}^k(R))$ and $\spec R\ne \V(J_n^k(R))$ as the prime ideal $(Y_1,\ldots,Y_{n+1})R$ of $R$ does not belong to $\V(J_n^k(R))$.
Fix $m\le n$.
Let $A=k[Z_1,\ldots,Z_m]/(Z_1^2,\ldots,Z_m^2)$ be a quotient of a polynomial ring over $k$.
We put $S:=R\otimes_k A\cong k[X,Y_1,\ldots,Y_{n+1},Z_1,\ldots,Z_m]/(XY_1,\ldots,XY_{n+1},Z_1^2,\ldots,Z_m^2)$.
The Jacobian matrix of $XY_1,\ldots,XY_{n+1},Z_1^2,\ldots,Z_m^2$ is the $(m+n+1)\times (m+n+2)$ matrix
\begin{align*}
U=
\begin{pmatrix}
   T & 0 & \cdots & 0\\
   0 & 2Z_1 & \ddots & \vdots \\
   \vdots & \ddots & \ddots & 0 \\
   0 & \cdots & 0 & 2Z_m \\
\end{pmatrix}.
\end{align*}
It is seen that $\codim_k A=m$ and $\codim_k S=m+1$ hold.
We have $J_{n-m}^k(S)=I_{n+1}(U)S\supseteq X^{n+1}S$ and $J_{n-m+1}^k(S)=I_{n+2}(U)S\subseteq (Z_1,\ldots,Z_m)S$.
This means $\sing S=\spec S\nsubseteq \V(J_{n-m}^k(S))$ and $\spec S= \V(J_{n-m+1}^k(S))$ as the prime ideal $(Y_1,\ldots,Y_{n+1}, Z_1,\ldots,Z_m)S$ of $S$ does not belong to $\V(J_{n-m}^k(S))$.
\end{ex}

Let $R$ be either an affine algebra over a field or an equicharacteristic complete local ring, of dimension $d$.
The purpose of the remainder of this section is to characterize $\edd R$ in terms of ideals generated by the minors of the Jacobian matrix.
The proposition below is a classical result known as a corollary of the Jacobian criterion; see \cite[Theorem 30.3]{Mat} for instance.
By \cite[Theorem 5.4]{W} and \cite[Theorem 3.7]{W3}, under some assumptions, and by \cite[Theorem 1.1]{IT2}, in the general case, it was proved that $\ca^{d+1}(R)$ contains some power of the Jacobian ideal of $R$ when $\edd R=0$.
This fact implies Proposition \ref{gen. of EandW} in the case $\edd R=0$; see \cite[Lemma 2.10(2)]{IT}.
(Indeed, thanks to Corollary \ref{SingV=(d+1)}, these assertions are now equivalent.)
The proofs of the above result in \cite{IT2} do not require the theory of smoothness.
Here, we provide an elementary proof of Proposition \ref{gen. of EandW} using only the result of Iyengar and Takahashi \cite{IT2}.

\begin{prop}\label{gen. of EandW}
Let $k$ be a field.
\begin{enumerate}[\rm(1)]
\item Let $R$ be an affine $k$-algebra.
Then $\sing R\subseteq \V(J_{\edd R}^k(R))$.
\item Let $R$ be an equicharacteristic complete local ring with residue field $k$.
Then $\sing R\subseteq \V(J_{\edd R}(R))$.
\end{enumerate}
\end{prop}

\begin{proof}
The proof is by induction on $e:=\edd R$.
When $e=0$, we have $J_0^k(R)=\jac_k(R)$.
It follows from \cite[Lemma 2.10(2)]{IT} and \cite[Theorem 1.1]{IT2} %\old{}
%typoがあるけどどうする? perfectを仮定すれば\cite[Corollary 16.20]{E}と\cite[Theorem 5.4]{W}で対応可能 
that $\sing R\subseteq\V(\ca(R))\subseteq \V(\jac_k(R))$ holds.

Assume $e>0$.
We only prove the affine case.
(The local case is proved similarly.)
We take a polynomial ring $A=k[X_1, \ldots,X_n]$ over $k$ and an ideal $I=(f_1,\ldots, f_r)$ of $A$ such that $A/I= R$.
Let $I=\q_1\cap\cdots\cap\q_m$ be a shortest primary decomposition of $I$ and $\sqrt{\q_i}=\p_i$ for each $1\le i\le m$.
We may assume that $\dim R/\q_1=\cdots=\dim R/\q_l=\dim R$, $\dim R/\q_i<\dim R$ for all $l+1\le i\le m$, and $\p_1,\cdots,\p_h$ are all the minimal prime ideal of $A/I$.
Note that $m\ge h>l$ as $e>0$.
Then $e={\rm sup}\{\dim R- \dim A/\p_i \mid l+1\le i\le h \}$.
Put $\a=\q_{l+1}\cap\cdots\cap\q_m=(g_1,\ldots,g_s)$ and $B=k[X_1, \ldots,X_n,Y]$, and let $J=(f_1,\ldots, f_r, Yg_1,\ldots, Yg_s)$ be an ideal of $B$.
For any ideals $\mathfrak{b}, \mathfrak{c}$ of $A$, it is easy to see that the equalities $(\mathfrak{b}, Y)B\cap(\mathfrak{c}, Y)B=(\mathfrak{b}\cap\mathfrak{c}, Y)B$, $\mathfrak{b}B\cap\mathfrak{c}B=(\mathfrak{b}\cap\mathfrak{c})B$, and $(\mathfrak{b}, Y)B\cap\mathfrak{c}B=(\mathfrak{b}\cap\mathfrak{c}, \mathfrak{c}Y)B$ hold.
(Note that $\mathfrak{b}B=\mathfrak{b}[Y]$ and $(\mathfrak{b}, Y)B=\mathfrak{b}+YB$). Therefore we have $(\q_1, Y)B\cap\cdots\cap(\q_l,Y)B\cap\q_{l+1}B\cap\cdots\cap\q_m B=(I, \a Y)B=(f_1,\ldots, f_r, Yg_1,\ldots, Yg_s)=J$.
On the other hand, for any ideal $\mathfrak{b}$ of $A$, $B/(\mathfrak{b}, Y)B$ and $B/\mathfrak{b}B$ are isomorphic to $A/\mathfrak{b}$ and $(A/\mathfrak{b})[Y]$, respectively.
In particular, $\dim B/(\mathfrak{b}, Y)B=\dim A/\mathfrak{b}$, and $\dim B/\mathfrak{b}B=\dim A/\mathfrak{b}+1$.
By this, we can easily see that $(\p_1, Y)B,\cdots,(\p_l, Y)B,\p_{l+1}B,\cdots,\p_h B$ are all the minimal prime ideal of $B/J$, and that the equalities $\dim (B/J)=\dim R$ and $\edd (B/J)=e-1$ hold.
Put $c=n-\dim R$.

Let $\p\in\sing R$. We prove that $\p$ belongs to $\V(J_e^k(R))$. 
The ideal $J_e^k(R)$ is generated by the $(c+e)$-minors of the Jacobian matrix $(\partial f_j/\partial X_i)$ of $f_1,\ldots,f_r$.
If $c+e>{\rm min}\{n,r\}$,  $J_e^k(R)=0$, which means $\p\in\spec R=\V(J_e^k(R))$. 
We may assume $c+e\le {\rm min}\{n,r\}$.
Let $P$ be a prime ideal of $A$ such that $\p=P/I$, and let $Q=PB+YB$ be an ideal of $B$.
Since $B/Q\cong A/P$, $Q$ is a prime ideal containing $J$.
Put $\q=Q/J$.
We show that $\q$ is in $\sing (B/J)$ by considering three cases.
First, if $Y$ belongs to $JB_Q$, then $JB_Q=(f_1,\ldots, f_r, Y)B_Q$.
So, $(B/J)_\q\cong B_Q/JB_Q\cong A_P/IA_P\cong R_\p$ and thus $(B/J)_\q$ is not regular.
In a second step, we assume $JB_Q=\a B_Q$. 
Then we can easily see that $IA_P=\a A_P$. 
Consider the following two commutative diagrams:
\begin{equation*}
\xymatrix{
A/I \ar[r]\ar@{->>}[d] & B/J\ar@{->>}[d]\\
A/\a \ar[r]& (A/\a)[Y]\cong B/\a B \quad {\rm and} 
}
\xymatrix{
R_\p=(A/I)_\p \ar[r]\ar[d]_{\rotatebox{90}{$\cong$}}  & (B/J)_\q \ar[d]_{\rotatebox{90}{$\cong$}} \\
A_P/\a A_P \ar[r]&  B_Q/\a B_Q.
}
\end{equation*}
The right side is a localization of the left side.
Since $A/\a\to B/\a B$ is flat, $R_\p\to (B/J)_\q$ is flat and local.
The ring $R_\p$ is not regular, neither is $(B/J)_\q$.
Finally suppose that $Y\notin JB_Q$ and $JB_Q\subsetneq \a B_Q$. 
Then $g_j\notin JB_Q$ for some $1\le j\le s$.
However, $Yg_j\in JB_Q$, which implies that $(B/J)_\q\cong B_Q/JB_Q$ is not an integral domain, and that it is not a regular local ring.

We obtain $\q\in\sing (B/J)$.
By the induction hypothesis, $\sing (B/J)\subseteq\V(J_{e-1}^k(B/J))$ as $\edd (B/J)=e-1$.
By the equality $n+1-\dim (B/J)=n+1-\dim R=c+1$, the ideal $J_{e-1}^k(B/J)$ is generated by the $(c+e)$-minors of the $(n+1)\times (r+s)$ matrix
\begin{align*}
\begin{pmatrix}
   \dfrac{\partial f_1}{\partial X_1} & \cdots &  \dfrac{\partial f_r}{\partial X_1} & \dfrac{\partial g_1}{\partial X_1} Y & \cdots  &  \dfrac{\partial g_s}{\partial X_1} Y  \\
   \vdots & \ddots & \vdots &  \vdots & \ddots & \vdots\\
   \dfrac{\partial f_1}{\partial X_n} & \cdots  &  \dfrac{\partial f_r}{\partial X_n} & \dfrac{\partial g_1}{\partial X_n} Y & \cdots  &  \dfrac{\partial g_s}{\partial X_n} Y\\
   0 & \cdots & 0 & g_1 & \cdots & g_s
\end{pmatrix}
\end{align*}
over $B/J$.
Noting that $c+e\le {\rm min}\{n,r\}$, $J_{e-1}^k(B/J)$ contains $I_{c+e}(\partial f_j/\partial X_i)(B/J)$, and thus so does $\q$.
This yields $I_{c+e}(\partial f_j/\partial X_i)B\subset Q=PB+YB=P+YB$ in $B$, which deduces $I_{c+e}(\partial f_j/\partial X_i)A\subset P$ in $A$.
Therefore we have $\p\in\V(I_{c+e}(\partial f_j/\partial X_i)R)=\V(J_e^k(R))$. 
\end{proof}

The result below is also a classical one; see \cite[Theorem 30.4]{Mat} for instance.
Using Theorems \ref{affine ne} and \ref{ct and t}, it can also be proven through elementary arguments.

\begin{prop}\label{new d+1}
Let $k$ be a field.
\begin{enumerate}[\rm(1)]
\item Let $R$ be an affine $k$-algebra.
Then $\spec R=\V(J_{\edd R+1}^k(R))$.
\item Let $R$ be an equicharacteristic complete local ring with residue field $k$.
Then $\spec R=\V(J_{\edd R+1}(R))$.
\end{enumerate}
\end{prop}

\begin{proof}
We only prove (1).
(The assertion (2) is proved similarly.)
Assume $\spec R\ne \V(J_{\edd R+1}^k(R))$.
Take a polynomial ring $S=k[X_1, \ldots,X_m]$ over $k$ and an ideal $I$ of $S$ such that $S/I= R$.
We put $e=\edd R$, $A=k[Y]/(Y^2)$, and $B=k[X_1, \ldots,X_m,Y]/(I, Y^2)\cong A\otimes_k R$.
Applying (6)$\Rightarrow$(1) of Theorem \ref{affine ne} to $n=e+1$, we obtain $\sing B\nsubseteq \V(J_e^k(B))$.
Thanks to Proposition \ref{gen. of EandW}, we get $\sing B\subseteq \V(J_{\edd B}^k(B))$, which means $\edd B>e$.
Set $T=k[X_1, \ldots,X_m,Y]$ and $J=(I, Y^2)T$.
There exists a minimal prime ideal $\p/J$ of $B$ such that $\dim B-\dim T/\p>e$, where $\p\in\spec T$.
Now $Y$ is in $\p$.
For $\q:=\p\cap S$, we have $J\subseteq \q T+YT\subseteq \p$.
The isomorphism $T/(\q T+YT) \cong S/\q$ says that $\q T+YT$ is a prime ideal of $T$.
The equality $\q T+YT=\p$ holds since $\p/J$ is a minimal prime ideal of $T/J$.
By Krull's height theorem, we have $\height \p/\q T\le 1$.
%On the other hand, similarly, $\q T$ is a prime ideal of $T$.
%Hence $\height \p/\q T=1$ as $Y$ is not in $\q T$.
Since the ring homomorphisms $S\to T$ and $R=S/I \to A\otimes_kB\cong T/J$ are flat, we have $\height\p=\height \q+\height \p/\q T\le \height \q+1$ and $0=\height(\p/J)\ge \height(\q/I)$; see \cite[Theorem 15.1]{Mat} for instance.
So, $\q/I$ is a minimal prime ideal of $R$ and thus $\dim R-\dim S/\q\le e$.
However $\dim R-\dim S/\q=(\dim B-\dim A)-(m-\height\q)\ge \dim B-(m+1-\height\p)=\dim B-\dim T/\p>e$, which is a contradiction.
\end{proof}

Propositions \ref{gen. of EandW} and \ref{new d+1} asserts that for any affine algebra $R$ over a field $k$ and integer $N\ge 0$, if $\edd R\le N$, then $\sing R\subseteq \V(J_n^k(R))$ and $\spec R=\V(J_{n+1}^k(R))$ hold for any $n\ge N$.
The example below says that the converse does not hold.

\begin{ex}\label{fixringnothold}
Let $S=k[X_1,\ldots X_n]$ be a polynomial ring over a field $k$, $I$ an ideal of $S$ and $m\ge 2$.
Put $R=S/I^m$ and $c=\codim_k R$.
Since $\partial f/\partial X_i$ belongs to $I^{m-1}$ for any $f\in I^m$ and $1\le i\le n$, we have $J_l^k(R)\subseteq I^{m-1}/I^m$ for any $l\ge 1-c$, which means $\spec R=\V(J_l^k(R))$.
On the other hand, $\edd R$ is not necessarily less than or equal to $1-c$ as $I$ is any ideal of $S$.
For example, if $R=k[X,Y,Z]/(XY,XZ)^2$, $\sing R=\spec R=\V(J_l^k(R))$ hold for all $l\ge 0$, but $\edd R=1$, which is greater than $0=1-\codim_k R$.
\end{ex}

The following theorem is the main result of this section.
By Example \ref{fixringnothold}, fixing the ring, the converses of Propositions \ref{gen. of EandW} and \ref{new d+1} does not hold in general.
As seen in Remark \ref{homeo}(2), $\edd R$ depends only on the spectrum of $R$.
From this perspective, the converses of Propositions \ref{gen. of EandW} and \ref{new d+1} also hold.

\begin{thm}\label{origin of main}
Let $k$ be a perfect field.
\begin{enumerate}[\rm(1)]
\item Let $R=k[X_1, \ldots,X_m]$ be a polynomial ring over $k$, and $V$ a closed subset of $\spec R$ of combinatorial dimension $d$.
For every integer $n\ge 0$, the following are equivalent:
\begin{enumerate}[\rm(i)]
\item For any mimimal element $\p$ of $V$, the inequality $d \le \dim R/\p+n$ holds;
\item For any ideal $I$ of $R$ satisfying $V=\V(I)$, $\sing (R/I) \subseteq \V(J_n^k(R/I))$ holds;
\item For any ideal $I$ of $R$ satisfying $V=\V(I)$, $\spec (R/I) = \V(J_n^k(R/I))$ holds;
\item For any ideal $I$ of $R$ satisfying $V=\V(I)$, $\V(\ca^{d+1} (R/I)) \subseteq \V(J_n^k(R/I))$ holds.
%there exists $l>0$ such that for all finitely generated $R/I$-modules $M,N$, $(J_{n+\codim_k (R/I)}^k(R/I))^l \Ext^{d+1}_{R/I}(M,N)=0$ holds.
\end{enumerate}
\item Let $R=k\llbracket X_1, \ldots,X_m\rrbracket$ be a formal power series ring over $k$, and $V$ a closed subset of $\spec R$ of combinatorial dimension $d$.
For every integer $n\ge 0$, the following are equivalent:
\begin{enumerate}[\rm(i)]
\item For any mimimal element $\p$ of $V$, the inequality $d\le \dim R/\p+n$ holds;
\item For any ideal $I$ of $R$ satisfying $V=\V(I)$, $\sing (R/I) \subseteq \V(J_n(R/I))$ holds;
\item For any ideal $I$ of $R$ satisfying $V=\V(I)$, $\spec (R/I) = \V(J_n(R/I))$ holds;
\item For any ideal $I$ of $R$ satisfying $V=\V(I)$, $\V(\ca^{d+1} (R/I)) \subseteq \V(J_n(R/I))$ holds.
%there exists $l>0$ such that for all finitely generated $R/I$-modules $M,N$, $(J_{n+\codim (R/I)}(R/I))^l \Ext^{d+1}_{R/I}(M,N)=0$ holds.
\end{enumerate}
\end{enumerate}
\end{thm}

\begin{proof}
Similarly, we only show the assertion (1).
For any ideal $I$ of $R$ such that $V=\V(I)$, $\dim (R/I)=d$ holds.
Thanks to Corollary \ref{SingV=(d+1)}, there is an equality $\sing (R/I)=\V(\ca^{d+1} (R/I))$, which induces (ii)$\Leftrightarrow$(iv).
Suppose that $d \le \dim R/\p+n$ holds for any mimimal element $\p$ of $V$.
Then $e=\edd (R/I)\le n$ for any ideal $I$ of $R$ satisfying $V=\V(I)$.
It follows from Remark \ref{J is w.d.}(3), Propositions \ref{gen. of EandW} and \ref{new d+1} that $\sing (R/I) \subseteq \V(J_e^k(R/I))\subseteq \V(J_n^k(R/I))$ and $\spec (R/I)=\V(J_{e+1}^k(R/I))\subseteq \V(J_{n+1}^k(R/I))$ hold.
Hence (i)$\Rightarrow$(ii) and (i)$\Rightarrow$(iii) hold.

To prove the converses, assume the existence of a mimimal element $\p$ of $V$ such that $d> \dim R/\p+n$.
Put $h=\height\p=m-\dim R/\p$.
Let $\p=\p_0,\p_1,\ldots,\p_s$ be all the minimal elements of $V$.
We see that $d=\dim R/\p_u$ for some $1\le u\le s$.
Then $\p\ne 0$ and $\p_u\ne 0$ since $d> \dim R/\p$ and $\p$ is minimal.
Therefore $h-n=m-\dim R/\p-n>m-d=m-\dim R/\p_u=\height\p_u>0$.
Since $R_\p$ is a regular local ring of dimension $h$, we can choose $f_1,\ldots,f_h\in\p$ such that $\height(f_1,\ldots,f_h)=h$ and $(f_1,\ldots,f_h)R_\p=\p R_\p$; see the proof of \cite[Propositon 4.4]{W}.
Put $J=(f_1,\ldots,f_h)$.
Then $\p/J\notin\sing(R/J)$ and $R/J$ is equidimensional as the heights of all the minimal prime ideals of $J$ are $h$ by Krull's height theorem.
Now $k$ is perfect.
It follows from \cite[Corollary 16.20]{E} (and in the local case, it follows from \cite[Lemma 4.3, and Propositions 4.4 and 4.5]{W}) that $\V(\jac_k(R/J))\subseteq \sing(R/J)$ and thus $\jac_k(R/J)\nsubseteq\p/J$.
We ontain $I_h(\partial f_j/\partial X_i)\nsubseteq\p$.
We take $\gamma\in \bigcap_{i=1}^s\p_i\setminus\p$.

We prove that condition (ii) does not hold.
Let $U_j$ be the $m\times (h-1)$ submatrix obtained by removing the $j$-th column of $(\partial f_j/\partial X_i)$, that is to say, $U_j$ is the Jacobian matrix of $f_1,\ldots,f_{i-1},f_{i+1},\ldots,f_h$.
By the Laplace expansion, we get $I_h(\partial f_j/\partial X_i)\subseteq \sum_{j=1}^h I_{h-1}(U_j)$.
We may assume $I_{h-1}(U_h)\nsubseteq\p$ as $I_h(\partial f_j/\partial X_i)\nsubseteq\p$.
Set $\q=(f_1,\ldots,f_{h-1})+\p^2$ and $I=\q\cap\p_1\cap\cdots\cap\p_s$.
Then $\V(I)=V$ and $\q R_\p\subsetneq \p R_\p$ by \cite[Corollary of Theorem 2.2]{Mat} and $(f_1,\ldots,f_{h-1})R_\p\subsetneq \p R_\p$.
There is an isomorphism $R_\p/I R_\p\cong R_\p/\q R_\p$, and the right side is an Artinian local ring that is not a field.
This means that $\p/I$ belongs to $\sing (R/I)$.
We put $\alpha_r=\gamma f_r\in I$ for each $1\le r\le h-1$.
Calculating the Jacobian matrix of the system of generators of $I$ obtained by extending $\alpha_1,\ldots,\alpha_{h-1}$, we see that $I_{n+m-d} (\partial \alpha_j/\partial X_i)(R/I)$ is contained in $J_n^k(R/I)$ since $h-1\ge n+m-d$.
On the other hand, for any $1\le r\le h-1$ and any $1\le t\le m$, the equality
$$
\dfrac{\partial \alpha_r}{\partial X_t}=\gamma \dfrac{\partial f_r}{\partial X_t}+\dfrac{\partial \gamma}{\partial X_t} f_r
$$
holds.
As $f_r$ are in $\p$ for all $1\le r\le h-1$, it is seen that $\gamma^{n+m-d} I_{n+m-d} (U_h)\subseteq I_{n+m-d} (\partial \alpha_j/\partial X_i)+\p$.
By $I_{n+m-d} (U_h)\supseteq I_{h-1}(U_h)\nsubseteq\p$ and $\gamma\notin\p$, 
We conclude that $I_{n+m-d} (\partial \alpha_j/\partial X_i)$ is not contained in $\p$, which yields $\p/I\notin \V(J_n^k(R/I))$.

Show that condition (iii) does not hold.
We put $J=\p\cap\p_1\cap\cdots\cap\p_s$ and $\beta_r=\gamma f_r$ for all $1\le r\le h$.
An analogous argument shows that $I_{n+m-d+1} (\partial \beta_j/\partial X_i)(R/J)$ is contained in $J_{n+1}^k(R/J)$ and that $\gamma^{n+m-d+1} I_{n+m-d+1} (\partial f_j/\partial X_i)$ is contained in $I_{n+m-d+1} (\partial \beta_j/\partial X_i)+\p$.
Noting that $I_h(\partial f_j/\partial X_i)$ is not contained in $\p$, and hence neither $I_{n+m-d+1} (\partial f_j/\partial X_i)$, similarly, we have $I_{n+m-d+1} (\partial \beta_j/\partial X_i)\nsubseteq\p$, which implies $\p/J\notin \V(J_{n+1}^k(R/J))$.
\end{proof}

The following result is a direct corollary of Theorem \ref{origin of main}.

\begin{cor}\label{main cor of edd}
Let $k$ be a perfect field and let $n\ge 0$ be an integer.
\begin{enumerate}[\rm(1)]
\item For an affine $k$-algebra $R$ of dimension $d$, the following are equivalent:
\begin{enumerate}[\rm(i)]
\item The inequality $\edd R\le n$ holds;
\item For any field $l$ and any affine $l$-algebra $S$ with $\spec S\cong\spec R$, $\sing S \subseteq \V(J_n^l(S))$ holds;
\item For any field $l$ and any affine $l$-algebra $S$ with $\spec S\cong\spec R$, $\spec S = \V(J_{n+1}^l(S))$ holds;
\item For any field $l$ and any affine $l$-algebra $S$ with $\spec S\cong\spec R$, $J_n^l(S) \subseteq \sqrt{\ca^{d+1}(S)}$ holds;
\item For any affine $k$-algebra $S$ with $\spec S\cong\spec R$, $\sing S \subseteq \V(J_n^k(S))$ holds;
\item For any affine $k$-algebra $S$ with $\spec S\cong\spec R$, $\spec S = \V(J_{n+1}^k(S))$ holds;
\item For any affine $k$-algebra $S$ with $\spec S\cong\spec R$, $J_n^k(S) \subseteq \sqrt{\ca^{d+1}(S)}$ holds.
\end{enumerate}
\item For an equicharacteristic complete local ring $(R,\m,k)$ of dimension $d$, the following are equivalent:
\begin{enumerate}[\rm(i)]
\item The inequality $\edd R\le n$ holds;
\item For any equicharacteristic complete local ring $S$ with $\spec S\cong\spec R$, $\sing S \subseteq \V(J_n(S))$;
\item For any equicharacteristic complete local ring $S$ with $\spec S\cong\spec R$, $\spec S = \V(J_{n+1}(S))$;
\item For any equicharacteristic complete local ring $S$ with $\spec S\cong\spec R$, $J_n(S) \subseteq \sqrt{\ca^{d+1}(S)}$;
\item For any equicharacteristic complete local ring $S$ with residue field $k$ such that $\spec S\cong\spec R$, $\sing S \subseteq \V(J_n(S))$ holds;
\item For any equicharacteristic complete local ring $S$ with residue field $k$ such that $\spec S\cong\spec R$, $\spec S = \V(J_{n+1}(S))$ holds;
\item For any equicharacteristic complete local ring $S$ with residue field $k$ such that $\spec S\cong\spec R$, $J_n(S) \subseteq \sqrt{\ca^{d+1}(S)}$ holds.
\end{enumerate}
\end{enumerate}
\end{cor}

\begin{proof}
We only show (i)$\Leftrightarrow$(ii)$\Leftrightarrow$(v) in (1).
Other implications are proven in a similar methods.
Suppose that $\edd R\le n$ holds.
Let $S$ be an affine algebra over a field $l$ such that $\spec S\cong\spec R$.
Remark \ref{homeo}(2) yields $\edd (A/I)=\edd R\le n$.
Proposition \ref{gen. of EandW} deduces that $\sing S \subseteq \V(J_n^l(S))$ holds.
We have (i)$\Rightarrow$(ii).
It is clear that (ii)$\Rightarrow$(v) holds.
Assume that $\sing S \subseteq \V(J_n^k(S))$ holds for any affine $k$-algebra $S$ such that $\spec S\cong\spec R$.
Take a polynomial ring $A=k[X_1, \ldots,X_m]$ over $k$ and an ideal $I$ of $A$ such that $A/I=R$.
If an ideal $J$ of $A$ satisfies $\V(J)=\V(I)$, then $\spec (A/J)\cong\V(J)=\V(I)\cong\spec (A/I)=\spec R$.
By assumption, $\sing (A/J) \subseteq \V(J_n^k(A/J))$ holds.
Theorem \ref{origin of main} says that for any mimimal element $\q$ of $\V(I)$, the inequality $\dim R \le \dim B/\q+n$ holds, which means $\edd R\le n$.
We obtain (v)$\Rightarrow$(i).
\end{proof}

In the case of $n=0$, the conditions (i), (iv), (v) and (vii) of Corollary \ref{main cor of edd} are rewritten as follows. 

\begin{cor}\label{equidimversionofmain}
Let $k$ be a perfect field.
\begin{enumerate}[\rm(1)]
\item For an affine $k$-algebra $R$ of dimension $d$, the following are equivalent:
\begin{enumerate}[\rm(i)]
\item $R$ is equidimensional;
\item For any affine algebra $S$ over a filed with $\spec S\cong\spec R$, there exists in integer $m>0$ such that $\jac_k(S)^m \Ext^{d+1}_{S}(M,N)=0$ holds for all $S$-modules $M,N$;
\item For any affine $k$-algebra $S$ with $\spec S\cong\spec R$, $\sing S=\V(\jac_k(S))$ holds;
\item For any affine $k$-algebra $S$ with $\spec S\cong\spec R$, $\sqrt{\ca^{d+1}(S)}=\sqrt{\jac_k(S)}$ holds.
\end{enumerate}
\item For an equicharacteristic complete local ring $(R,\m,k)$ of dimension $d$, the following are equivalent:
\begin{enumerate}[\rm(i)]
\item $R$ is equidimensional;
\item For any equicharacteristic complete local ring $S$ with $\spec S\cong\spec R$, there exists $m>0$ such that $\jac(S)^m \Ext^{d+1}_{S}(M,N)=0$ holds for all $S$-modules $M,N$;
\item For any equicharacteristic complete local ring $S$ with residue field $k$ such that $\spec S\cong\spec R$, $\sing S=\V(\jac(S))$ holds;
\item For any equicharacteristic complete local ring $S$ with residue field $k$ such that $\spec S\cong\spec R$, $\sqrt{\ca^{d+1}(S)}=\sqrt{\jac(S)}$ holds.
\end{enumerate}
\end{enumerate}
\end{cor}

\begin{proof}
Applying Corollary \ref{main cor of edd} to $n=0$, the implications (ii)$\Rightarrow$(i)$\Leftarrow$(iii) hold.
An analogous argument to the former part of the proof of Theorem \ref{origin of main} shows $\sing (S)=\V(\ca^{\dim S+1} (S))$, where $S$ is either an affine $k$-algebra or an equicharacteristic complete local ring with residue field $k$.
So, we have (iii)$\Leftrightarrow$(iv).
We see by Remark \ref{homeo}(2) and \cite[Theorem 1.1]{IT2} that (i)$\Rightarrow$(ii) holds.
Suppose that $R$ is an equidimensional equicharacteristic complete local ring with residue field $k$.
It follows from Remark \ref{homeo}(2) and \cite[Lemma 4.3, and Propositions 4.4 and 4.5]{W} that $\sing (S)\supseteq V(\jac S)$ holds for any equicharacteristic complete local ring $S$ with residue field $k$ such that $\spec S\cong\spec R$.
By this and Corollary \ref{main cor of edd}, we have  $\sing (S)=V(\jac S)$. 
If $R$ is an equidimensional affine $k$-algebra, the equality $\sing (S)=\V(\jac_k S)$ follows from \cite[Corollary 16.20]{E}.
Thus we get (i)$\Rightarrow$(iii).
\end{proof}

We close this section by providing an example related to Corollary \ref{main cor of edd}.

\begin{ex}
Let $R=k[X,Y_1,\ldots,Y_{n+1}]/(XY_1,\ldots,XY_{n+1})$ be a quotient of a polynomial ring over a perfect field $k$, where $n\ge 1$.
Then $\edd R=n$.
We see that $\sing R\subseteq\V(J_n^k(R))$ and $\spec R=\V(J_{n+1}^k(R))$ hold; see Example \ref{nhensu example}.
According to Corollary \ref{main cor of edd}, there exist affine $k$-algebras $S$ and $T$ such that $\sing S\nsubseteq\V(J_{n-1}^k(S))$ and $\spec T\ne \V(J_n^k(T))$.
By Example \ref{nhensu example}, one has $\spec R\ne \V(J_n^k(R))$.
So $T$ can be chosen as $R$ itself.
On the other hand, we get $\sing R\subseteq\V(J_{n-1}^k(R))$ by Example \ref{nhensu example}.
Now let $S=k[X,Y_1,\ldots,Y_{n+1}]/(XY_1^2,XY_2,\ldots,XY_{n+1})$ be a quotient of a polynomial ring over $k$.
It is seen that $\spec S\cong\spec R$ and $\sing S=\V((Y_1,\ldots,Y_{n+1})S)$.
The ideal $J_{n-1}^k(S)$ is generated by the $n$-minors of 
\begin{align*}
\begin{pmatrix}
   Y_1^2 & 2XY_1 & 0 & \cdots & 0 \\
   Y_2 & 0 & X & \ddots & \vdots \\
   \vdots & \vdots & \ddots & \ddots & 0 \\
   Y_{n+1} & 0 & \cdots & 0 & X \\
\end{pmatrix},
\end{align*}
which contains $X^nS$.
We conclude that $\sing S\nsubseteq\V(J_{n-1}^k(S))$ as $(Y_1,\ldots,Y_{n+1})S$ does not contain $X^nS$.
\end{ex}

\section{Uniqueness of the Jacobian ideal}

This section focuses on whether the Jacobian ideal is uniquely determined.
In the case of a complete local ring in equicharacteristic, when there are two presentations of the ring, the Jacobian ideal can be directly computed through an isomorphism.
After presenting two lemmas, we will prove this.

\begin{lem}\label{matrixlemma}
Let $R$ be a ring, $m, n\ge 0$, and $A=(a_{ij})_{1\le i\le m, 1\le j\le n}$ an $m\times n$ matrix over $R$.
Let 
\begin{align*}
B=
\begin{pmatrix}
   a_{11} & \cdots &  a_{1n} & \sum_{j=1}^n c_j a_{1j}+ b_1\\
   \vdots & \ddots & \vdots & \vdots \\
   a_{m1} & \cdots &  a_{mn} & \sum_{j=1}^n c_j a_{mj}+ b_m \\
\end{pmatrix}
\end{align*}
be an $(m+1)\times n$ matrix, where $b_i,c_j\in R$ for all $1\le i\le m$ and all $1\le j\le n$.
Then for any integer $r$, $I_r(A)\subseteq I_r(B)\subseteq I_r(A)+(b_1,\ldots,b_m)R$.
\end{lem}

\begin{proof}
It is clear that $I_r(A)\subseteq I_r(B)$.
When $r\le 1$, we can easily see that $I_r(B)\subseteq I_r(A)+(b_1,\ldots,b_m)R$. 
Assume $r\ge 2$.
To prove lemma, we only need to consider the $r$-th minors that include the $(n+1)$-th column.
That is, we have only to show that for all $1\le i_1<\cdots<i_r\le m$ and $1\le j_1<\cdots<j_{r-1}<j_r=n+1$, the element $\sum_{\sigma\in S_r} (\operatorname{sgn}\sigma) \prod_{l=1}^{r-1} a_{i_{\sigma (l)} j_l} 
\cdot (\sum_{k=1}^n c_k a_{i_{\sigma (r)} k}+ b_{i_{\sigma (r)}})$ belongs to $I_r(A)+(b_1,\ldots,b_m)R$.
Fix $1\le k\le n$.
If $k=j_l$ for some $1\le l\le r-1$, $(\operatorname{sgn}\sigma) \prod_{l=1}^{r-1} a_{i_{\sigma (l)} j_l}  a_{i_{\sigma (r)} k}+(\operatorname{sgn}\tau) \prod_{l=1}^{r-1} a_{i_{\tau (l)} j_l}  a_{i_{\tau (r)} k}=0$ for any $\sigma\in S_r$ and $\tau=\sigma \circ (l\ r)$.
Therefore we have 
\begin{align*}
& \sum_{\sigma\in S_r} (\operatorname{sgn}\sigma) \prod_{l=1}^{r-1} a_{i_{\sigma (l)} j_l} 
\cdot (\sum_{k=1}^n c_k a_{i_{\sigma (r)} k}+ b_{i_{\sigma (r)}}) \\
=& \sum_{k=1}^n c_k \sum_{\sigma\in S_r} (\operatorname{sgn}\sigma) \prod_{l=1}^{r-1} a_{i_{\sigma (l)} j_l}  a_{i_{\sigma (r)} k} + 
\sum_{\sigma\in S_r} (\operatorname{sgn}\sigma) \prod_{l=1}^{r-1} a_{i_{\sigma (l)} j_l}   b_{i_{\sigma (r)}} \\
=& \sum_{\substack{1\le k\le n \\ k\notin\{j_1,\ldots,j_{r-1}\}}} c_k \sum_{\sigma\in S_r} (\operatorname{sgn}\sigma) \prod_{l=1}^{r-1} a_{i_{\sigma (l)} j_l}  a_{i_{\sigma (r)} k} + 
\sum_{\sigma\in S_r} (\operatorname{sgn}\sigma) \prod_{l=1}^{r-1} a_{i_{\sigma (l)} j_l}   b_{i_{\sigma (r)}} \\
\in &\  (c_1,\ldots,c_n) I_r(A)+(b_1,\ldots,b_m)R.
\end{align*}
The proof is now completed.
\end{proof}

\begin{lem}\label{independent of generator}
Let $R$ be a ring, and let $S$ be either the polynomial ring $R[X_1, \ldots,X_l]$ or the formal power series ring $R\llbracket X_1, \ldots,X_l\rrbracket$.
Let $f_1,\ldots,f_m,g_1,\ldots,g_n\in S$.
If $I=(f_1,\ldots,f_m)=(g_1,\ldots,g_n)$, then $I_r(\partial f_j/\partial X_i)+I=I_r(\partial g_j/\partial X_i)+I$ for any integer $r$.
\end{lem}

\begin{proof}
Since $(f_1,\ldots,f_m)=(f_1,\ldots,f_m, g_1)=\cdots=(f_1,\ldots,f_m, g_1,\ldots,g_{n-1})=(f_1,\ldots,f_m, g_1,\ldots,g_n)=(f_2,\ldots,f_m, g_1,\ldots,g_n)=\cdots=(f_m, g_1,\ldots,g_n)=(g_1,\ldots,g_n)$, we may assume that $n=m+1$ and $f_i=g_i$ for any $1\le i\le m$.
Put $g_{m+1}=\sum_{i=1}^m a_i f_i$.
Then
\begin{align*}
\left(\frac{\partial g_j}{\partial X_i}\right)=
\begin{pmatrix}
   \dfrac{\partial f_1}{\partial X_1} & \cdots &  \dfrac{\partial f_m}{\partial X_1} & \sum\limits_{i=1}^m a_i \dfrac{\partial f_i}{\partial X_1} + \sum\limits_{i=1}^m  \dfrac{\partial a_i}{\partial X_1}f_i \\
   \vdots & \ddots & \vdots & \vdots \\
   \dfrac{\partial f_1}{\partial X_n} & \cdots &  \dfrac{\partial f_m}{\partial X_n} & \sum\limits_{i=1}^m a_i \dfrac{\partial f_i}{\partial X_n} + \sum\limits_{i=1}^m  \dfrac{\partial a_i}{\partial X_n}f_i  \\
\end{pmatrix}.
\end{align*}
By Lemma \ref{matrixlemma}, we have $I_r(\partial f_j/\partial X_i)\subseteq I_r(\partial g_j/\partial X_i)\subseteq I_r(\partial f_j/\partial X_i)+I$.
\end{proof}

It is well known that the Jacobian ideal of an affine algebra over a field is well-defined, as shown in the argument of Remark \ref{J is w.d.}.
In the case of complete equicharacteristic local rings, the Jacobian ideal can be computed and compared by using the properties of the local ring.

\begin{prop}\label{local unique}
Let $R$ be an equicharacteristic complete local ring with residue field $k$.
Suppose that $F:A:=k\llbracket X_1, \ldots,X_m\rrbracket/(f_1,\ldots,f_c) \to R$ and $G:B:=\llbracket Y_1, \ldots,Y_n\rrbracket/(g_1,\ldots,g_d) \to R$ are ring isomorphisms.
Then $F(I_{m+r}(\partial f_j/\partial X_i)A)=G(I_{n+r}(\partial g_j/\partial Y_i)B)$ holds for any integer $r$.
\end{prop}

\begin{proof}
Put $\psi=G^{-1}F$, $I=(f_1,\ldots,f_c)$, and $J=(g_1,\ldots,g_d)$.
Let $e$ be the embedding dimension of $R$, which is equal to the dimensions of $(X_1, \ldots,X_m)/(I+(X_1, \ldots,X_m)^2)$ and $(Y_1, \ldots,Y_n)/(J+(Y_1, \ldots,Y_n)^2)$ as $k$-vector space.
Then $m,n\ge e$, and $m=e$ holds if and only if $I$ is contained in $(X_1, \ldots,X_m)^2$.
We prove $\psi(I_{m+r}(\partial f_j/\partial X_i)A)=I_{n+r}(\partial g_j/\partial Y_i)B$ by induction on $m+n$.

First we deal with the case $m=n=e$.
For any $1\le i\le m$, we put $\psi(X_i +I)=\alpha_i+J$, where $\alpha_i\in (Y_1, \ldots,Y_n)$.
There is a ring homomorphism $\phi:S:=k\llbracket X_1, \ldots,X_m\rrbracket \to T:=\llbracket Y_1, \ldots,Y_n\rrbracket:X_i\mapsto\alpha_i$.
Then $\phi(X_1, \ldots,X_m)+J=(Y_1, \ldots,Y_n)$ as $\psi$ is an isomorphism.
Since $n=e$, we have $J\subseteq(Y_1, \ldots,Y_n)^2$, and hence $\phi(X_1, \ldots,X_m)=(Y_1, \ldots,Y_n)$ by \cite[Corollary of Theorem 2.2]{Mat}.
This says that $\phi$ is surjective and so it is an isomorphism as $m=n$.
The equality $(\phi(f_1),\ldots,\phi(f_c))=(g_1,\ldots,g_d)$ holds.
Thanks to Lemma \ref{independent of generator}, we get $I_{n+r}(\partial \phi(f_j)/\partial Y_i)B=I_{n+r}(\partial g_j/\partial Y_i)B$.
One has the equality
$$
\delta_{ij}=\frac{\partial Y_j}{\partial Y_i}=\frac{\partial \phi(\phi^{-1}(Y_j))}{\partial Y_i}=\sum_{t=1}^m \phi\left(\frac{\partial \phi^{-1}(Y_j)}{\partial X_t}\right) \frac{\partial \phi(X_t)}{\partial Y_i}
$$
where $\delta_{ij}$ is the Kronecker delta.
Thus the matrix $(\partial \phi(X_j)/\partial Y_i)$ invertible.
The equality
$$
\frac{\partial \phi(f_j)}{\partial Y_i}=\sum_{t=1}^m \phi\left(\frac{\partial f_j}{\partial X_t}\right) \frac{\partial \phi(X_t)}{\partial Y_i}
$$
induces the following commutative diagram:
\[
  \xymatrix@C=50pt@R=30pt{
T^{\oplus c} \ar[r]^{\left(\phi\bigl(\tfrac{\partial f_j}{\partial X_i}\bigr)\right)}\ar@{=}[d] & T^{\oplus m} \ar[d]_{\rotatebox{90}{$\cong$}}^{\bigl(\tfrac{\partial \phi(X_j)}{\partial Y_i}\bigr)} \\
T^{\oplus c} \ar[r]^{\bigl(\tfrac{\partial \phi(f_j)}{\partial Y_i}\bigr)}& T^{\oplus n}.
}
\]
Hence $\phi(I_{m+r}(\partial f_j/\partial X_i))=I_{n+r}(\partial \phi(f_j)/\partial Y_i)$
because these are Fitting invariants of the same module.
We obtain $\psi(I_{m+r}(\partial f_j/\partial X_i)A)=I_{n+r}(\partial \phi(f_j)/\partial Y_i)B=I_{n+r}(\partial g_j/\partial Y_i)B$.

Next, we handle the case $m+n>2e$.
We may assume $m>e$.
There is $h\in I\setminus (X_1, \ldots,X_m)^2$
For some $1\le l\le m$, the coefficient of $X_i$ in $h$ is not zero.
By this, $X_1,\ldots X_{l-1}, h, X_{l+1} \ldots,X_m$ is a regular system of parameters of $k\llbracket X_1, \ldots,X_m\rrbracket$ and $x:=\partial h/\partial X_l$ is unit.
There is a natural isomomorphism $\chi: k\llbracket X_1,\ldots X_{l-1}, X_{l+1} \ldots,X_m\rrbracket \to k\llbracket X_1, \ldots,X_m\rrbracket/(h)$.
We put $\chi^{-1}(I/(h))=(h_1,\ldots, h_e)$, $\chi$ induces an  isomomorphism $\chi: C:=k\llbracket X_1,\ldots X_{l-1}, X_{l+1} \ldots,X_m\rrbracket/(h_1,\ldots, h_e)\to k\llbracket X_1, \ldots,X_m\rrbracket/(h_1,\ldots, h_e, h)=A$.
By the induction hypothesis, we have $(\psi\circ\chi)(I_{m-1+r}(U)C)=I_{n+r}(\partial g_j/\partial Y_i)B$, where 
\begin{align*}
U=
\begin{pmatrix}
   \dfrac{\partial h_1}{\partial X_1} & \cdots & \dfrac{\partial h_e}{\partial X_1} \\
   \vdots & \ddots & \vdots \\
   \dfrac{\partial h_1}{\partial X_{l-1}} & \cdots & \dfrac{\partial h_e}{\partial X_{l-1}} \\
   \dfrac{\partial h_1}{\partial X_{l+1}} & \cdots & \dfrac{\partial h_e}{\partial X_{l+1}} \\
   \vdots & \ddots & \vdots  \\
   \dfrac{\partial h_1}{\partial X_m} & \cdots & \dfrac{\partial h_e}{\partial X_m} \\
\end{pmatrix}.
\end{align*}
On the other hand, it follows from $I=(h_1,\ldots, h_e, h)$ and Lemma \ref{independent of generator} that $I_{m+r}(\partial f_j/\partial X_i)$ is equal to the ideal generated by the $(m+r)$-minors of
\begin{align*}
\begin{pmatrix}
   \dfrac{\partial h_1}{\partial X_1} & \cdots & \dfrac{\partial h_e}{\partial X_1} & \dfrac{\partial h}{\partial X_1} \\
   \vdots & \ddots & \vdots \\
   \dfrac{\partial h_1}{\partial X_{l-1}} & \cdots & \dfrac{\partial h_e}{\partial X_{l-1}}  & \dfrac{\partial h}{\partial X_{l-1}} \\
   0 & \cdots & 0 & x \\
   \dfrac{\partial h_1}{\partial X_{l+1}} & \cdots & \dfrac{\partial h_e}{\partial X_{l+1}} & \dfrac{\partial h}{\partial X_{l+1}} \\
   \vdots & \ddots & \vdots  \\
   \dfrac{\partial h_1}{\partial X_m} & \cdots & \dfrac{\partial h_e}{\partial X_m} & \dfrac{\partial h}{\partial X_m} \\
\end{pmatrix},
\end{align*}
which is equal to $I_{m-1+r}(U)$.
The proof is now completed.
\end{proof}

\begin{rem}
Let $K$ and $L$ be a field, and $R$ a ring.
Assume that $R$ is finitely generated over both $K$ and $L$.
In this case, a natural question arises as to whether $J_n^K(R)=J_n^L(R)$ holds for any integer $n$.
For example, let $R=\mathbb{C}[X_1,\dots,X_m]/(f_1,\ldots,f_n)$ be a quotient of a polynomial ring over the field $\mathbb{C}$ of complex numbers.
The natural surjection $\phi:\mathbb{R}[X_1,\dots,X_m, Y]\to \mathbb{C}[X_1,\dots,X_m]$ such that $\phi(Y)=i$ induces an isomorphism $\mathbb{R}[X_1,\dots,X_m, Y]/(g_1,\ldots,g_n, Y^2+1)\cong R$, where $g_1,\ldots,g_n$ are representatives of the inverse images of $f_1,\ldots,f_n$.
The Jacobian matrix of $g_1,\ldots,g_n, Y^2+1$ is of the form
\begin{align*}
A:=
\begin{pmatrix}
   B & 0 \\
   \ast & 2Y \\
\end{pmatrix}.
\end{align*}
Then $\phi(B)$ is the Jacobian matrix of $f_1,\ldots,f_n$.
As $2Y$ is unit in $\mathbb{R}[X_1,\dots,X_m, Y]/(g_1,\ldots,g_n, Y^2+1)$, we have $J_n^{\mathbb{R}}(R)=\phi(I_{n+m+1-\dim R}(A))R=\phi(I_{n+m-\dim R}(B))R=I_{n+m-\dim R}(\phi(B))R=J_n^{\mathbb{C}}(R)$ for any $n$.
Moreover, a similar argument shows that $J_n^{K}(R)=J_n^{L}(R)$ holds for any $n$, any finite (simple) extension $L/K$, and any affine $L$-algebra $R$.
However, in general, it is unknown whether equality holds.
\end{rem}

%%%%%%%%%%%%%%%%%%%%%%%%%%%%%%%%%%%%%%%%%%%%%%
\begin{ac}
The author would like to thank his supervisor Ryo Takahashi, and Yuki Mifune, for providing the motivation for this research.
The author is also grateful to Linquan Ma, Yuya Otake, Ryo Takahashi, and Tatsuki Yamaguchi for their valuable comments.
\end{ac}
%%%%%%%%%%%%%%%%%%%%%%%%%%%%%%%%%%%%%%%%%%%%%%%%%%%%%%%%%%%%%

\end{document}